\documentclass[11pt]{amsart}

\usepackage{geometry, amsmath, amssymb, esint}

\usepackage[colorlinks=true,urlcolor=blue,
citecolor=red,linkcolor=blue,linktocpage,pdfpagelabels,
bookmarksnumbered,bookmarksopen]{hyperref}
\usepackage[hyperpageref]{backref}

\newtheorem{teo}{Theorem}[section]
\newtheorem{lm}[teo]{Lemma}

\newtheorem{prop}[teo]{Proposition}

\newtheorem{coro}[teo]{Corollary}
\newtheorem*{HKS}{Nonlocal Hong-Krahn-Szego inequality}
\newtheorem*{second}{The second eigenvalue}

\theoremstyle{definition}
\newtheorem{oss}[teo]{Remark}
\newtheorem*{ack}{Acknowledgements}

\numberwithin{equation}{section}

\title{The second eigenvalue of the fractional $p-$Laplacian}
\author[Brasco]{Lorenzo Brasco}
\address{Aix-Marseille Universit\'e, CNRS, Centrale Marseille, I2M, UMR 7373, 39 Rue Fr\'ed\'eric Joliot Curie, 13453 Marseille, France}
\email{lorenzo.brasco@univ-amu.fr}
\email{enea.parini@univ-amu.fr}

\author[Parini]{Enea Parini}


\begin{document}

\begin{abstract}
We consider the eigenvalue problem for the {\it fractional $p-$Laplacian}
in an open bounded, possibly disconnected set $\Omega \subset \mathbb{R}^n$, under homogeneous Dirichlet boundary conditions. After discussing some regularity issues for  eigenfunctions, we show that the second eigenvalue $\lambda_2(\Omega)$ is well-defined, and we characterize it by means of several equivalent variational formulations. In particular, we extend the mountain pass characterization of Cuesta, De Figueiredo and Gossez to the nonlocal and nonlinear setting. Finally, we consider the minimization problem
\[ 
\inf \{\lambda_2(\Omega)\,:\,|\Omega|=c\}. 
\]
We prove that, differently from the local case, an optimal shape does not exist, even among disconnected sets. A minimizing sequence is given by the union of two disjoint balls of volume $c/2$ whose mutual distance tends to infinity.
\end{abstract}

\maketitle
\begin{center}
\begin{minipage}{11cm}
\small
\tableofcontents
\end{minipage}
\end{center}

\section{Introduction}
\label{sec:1}

\subsection{Overview and aim of the paper}
Let $\Omega\subset\mathbb{R}^N$ be a bounded open set, $0<s<1$ and $1<p<\infty$.
This paper is concerned with the nonlinear and nonlocal Dirichlet eigenvalue problem
\begin{equation} 
\label{sfeigenvalueintro} 
(-\Delta_p)^s u = \lambda\, |u|^{p-2}\, u,\ \ \mbox{ in }\Omega,\quad\qquad u = 0\ \ \mbox{ in }\mathbb{R}^N \setminus \Omega,
\end{equation}
where
\[ 
(-\Delta_p)^s \,u (x) := 2\,\lim_{\delta\to 0^+}\int_{\{y\in\mathbb{R}^N\, :\, |y-x|\ge \delta\}} \frac{|u(x)-u(y)|^{p-2} \,(u(x)-u(y))}{|x-y|^{N+s\,p}}\,dy,
\]
is the \emph{fractional $p-$Laplacian}. Here solutions of \eqref{sfeigenvalueintro} are always understood in the weak sense, see equation \eqref{wfeigenvalue} below.
\par
If $\lambda\in\mathbb{R}$ is such that \eqref{sfeigenvalueintro} admits a solution $u\not\equiv0$, then we say that $\lambda$ is an {\it $(s,p)-$eigenvalue} of $\Omega$. Correspondingly, $u$ is an {\it $(s,p)-$eigenfunction} associated to $\lambda$. The eigenvalue problem \eqref{sfeigenvalueintro} was first introduced by Lindgren and Lindqvist in \cite{LL} and investigated by several authors afterwards, we cite for example \cite{BLP,FP} and \cite{IS}. 
\par 
Observe that for $p=2$ the operator $(-\Delta_p)^s$ reduces to the well-known fractional Laplacian, which has been extensively studied in the last years (see for example \cite{CS,FV} and the references therein). We point out that the terminology {\it fractional $p-$Laplacian} is not standard, though somehow justified by the fact that the operator $(-\Delta_p)^s$ arises as the first variation of the fractional Dirichlet integral
\begin{equation}
\label{phiintro}
\Phi_{s,p}(u)=\int_{\mathbb{R}^N}\int_{\mathbb{R}^N} \frac{|u(x)-u(y)|^p}{|x-y|^{N+s\,p}}\, dx\,dy,
\end{equation}
and therefore it is the nonlocal counterpart of the usual $p-$Laplacian operator 
\[
-\Delta_p u :=-\mathrm{div}\left(|\nabla u|^{p-2}\,\nabla u\right).
\]
By homogeneity, it is not difficult to see that $(s,p)-$eigenvalues correspond to critical points of the functional \eqref{phiintro} restricted to the manifold $\mathcal{S}_p(\Omega)$ consisting of functions having unitary $L^p$-norm. \par
In order to put the studies of this paper in the right framework, let us start by recalling some known facts about the eigenvalue problem \eqref{sfeigenvalueintro}. First of all, in analogy with the local case, the {\it spectrum of the fractional $p-$Laplacian}, i.e the set 
\[
\sigma_{s,p}(\Omega)=\{\lambda\in\mathbb{R}\, :\, \lambda \mbox{  is an $(s,p)-$eigenvalue}\}
\]
is a closed set (see \cite{LL}). Moreover, it is possible to define the first eigenvalue $\lambda_1(\Omega)$, i.e. the smallest $\lambda\in\sigma_{s,p}(\Omega)$. The first eigenvalue has a variational characterization, as it corresponds to the minimum of $\Phi_{s,p}$ on $\mathcal{S}_p(\Omega)$. In other words, $\lambda_1(\Omega)$ coincides with the sharp constant in the following Poincar\'e inequality (see \cite[Lemma 2.4]{BLP})
\[
c\, \int_\Omega |u|^p\, dx\le \int_{\mathbb{R}^N}\int_{\mathbb{R}^N} \frac{|u(x)-u(y)|^p}{|x-y|^{N+s\,p}}\, dx\,dy,\qquad u\in C^\infty_0(\Omega).
\]
The quantity $\lambda_1(\Omega)$ can be estimated from below in terms of $|\Omega|$ in a sharp way, exactly as in the local case. This is a consequence of the {\it Faber-Krahn inequality} (see \cite[Theorem 3.5]{BLP})
\[
\lambda_1(\Omega)\ge \left(\frac{|B|}{|\Omega|}\right)^\frac{s\,p}{N}\, \lambda_1(B),
\]
where $B$ is any $N-$dimensional ball. Equality in the previous holds if and only if $\Omega$ itself is a ball. The result can also be rephrased by saying:
\begin{center}
``{\it Among all domains of fixed volume, the ball has the smallest first eigenvalue.}'' \end{center}
The main aim of this paper is to study the {\it second eigenvalue $\lambda_2(\Omega)$} of the fractional $p-$Laplacian, for every $1<p<\infty$ and $0<s<1$. Observe that since we are dealing with a {\it nonlinear} eigenvalue problem, actually it is not even clear that it is possible to speak about a second eigenvalue. Indeed, the spectrum $\sigma_{s,p}(\Omega)$ in principle could contain a sequence accumulating to $\lambda_1(\Omega)$. 
\par
Thus, at first we want to show that $\lambda_2(\Omega)$ is well-defined and give a variational characterization for it. Then we want to prove sharp lower bounds for $\lambda_2(\Omega)$ in terms of $|\Omega|$, similarly to the Faber-Krahn inequality. It is useful to recall at this point that in the local case this is indeed possible, thanks to the {\it Hong-Krahn-Szego inequality}. This asserts that:
\begin{center}
``{\it Among all domains of fixed volume, the disjoint union of two equal balls}\\
{\it has the smallest second eigenvalue.}'' \end{center}
In scaling invariant form, this reads as
\[
\lambda_2(\Omega)\ge \left(\frac{2\,|B|}{|\Omega|}\right)^\frac{p}{N}\, \lambda_1(B),
\]
where $B$ is again any $N-$dimensional ball. Equality holds if and only if $\Omega$ itself is a disjoint union of two equal balls.
For $p=2$ the result was proved long time ago by Krahn \cite{K}. Then this has been probably neglected and rediscovered some years later by Hong \cite{Ho} and P. Szego \cite{Po}. The general case of the $p-$Laplacian has been recently addressed in \cite[Theorem 3.2]{BF_MM}.

\subsection{Results of the paper}

The first main result of this paper is the following (see Sections \ref{sec:4} and \ref{sec:5} for the precise statements). 
\begin{second}
Let $1<p<\infty$ and $0<s<1$. Let $\Omega\subset\mathbb{R}^N$ be an open and bounded set. There exists a real positive number $\lambda_2(\Omega)$ with the following properties:
\begin{itemize}
\item $\lambda_2(\Omega)$ is an $(s,p)-$eigenvalue;
\vskip.2cm
\item $\lambda_2(\Omega)>\lambda_1(\Omega)$; 
\vskip.2cm
\item if $\lambda>\lambda_1(\Omega)$ is an $(s,p)-$eigenvalue, then $\lambda\ge \lambda_2(\Omega)$.
\end{itemize}
Moreover, it has the following variational characterization
\begin{equation}
\label{scampagnata}
\lambda_2(\Omega)=\inf_{\gamma\in \Gamma(u_1,-u_1)}\,\max_{u\in \mathrm{Im}(\gamma)} \Phi_{s,p}(u),
\end{equation}
where $\Gamma(u_1,-u_1)$ is the set of continuous paths on $\mathcal{S}_p(\Omega)$ connecting the first eigenfunction $u_1$ to its opposite $-u_1$.
\end{second} 
The eigenvalue $\lambda_2(\Omega)$ will be constructed by means of a variational minimax procedure originally introduced by Dr\'abek and Robinson in \cite{DR}. In particular, in this paper we will deliberately avoid to use any index theory. 
The mountain pass characterization \eqref{scampagnata} is the nonlocal counterpart of the result by Cuesta, De Figueiredo and Gossez for the local case, see \cite{CDG2}. We point out that our proof differs from that of \cite{CDG2} and is similar to that of \cite[Proposition 5.4]{BF_nodea}, which is based on the so called {\it hidden convexity} for Dirichlet integrals (see \cite{BF}). 
\vskip.2cm
On the contrary for the Hong-Krahn-Szego inequality the situation in the local case and in the nonlocal one are quite different.
Due to the nonlocal effects, the mutual position of the connected components of the domain influences the spectrum $\sigma_{s,p}(\Omega)$ of the operator. More important, as already observed in \cite{LL}, if $u$ is a sign-changing eigenfunction with eigenvalue $\lambda$ and $\Omega_+$ is one of its nodal domains, {\it it is no more true that} $\lambda_1(\Omega_+)=\lambda$ (see Lemma \ref{lm:LL} below). Also, the restriction of $u$ to $\Omega_+$ in general is not a first eigenfunction of $\Omega_+$. 
{\it This point marks a huge difference with the local case}.
We then have the following sharp estimate for $\lambda_2(\Omega)$, which is partially in contrast with the local case. This is the second main result of the paper.
\begin{HKS}[Nonlocal Hong-Krahn-Szego inequality]
Let $1<p<\infty$ and $0<s<1$. For every $\Omega\subset\mathbb{R}^N$ open and bounded set we have
\begin{equation}
\label{HKSintro}
\lambda_2(\Omega)>\left(\frac{2\,|B|}{|\Omega|}\right)^\frac{s\,p}{N}\, \lambda_1(B),
\end{equation}
where $B$ is any $N-$dimensional ball. Equality is never attained in \eqref{HKSintro}, but the estimate is sharp in the following sense: if $\Omega_n$ is a disjoint union of two equal balls $B_R(x_n)$ and $B_R(y_n)$ such that  
\[
\lim_{n\to\infty} |x_n-y_n|=+\infty,
\]
then
\[
\lim_{n\to\infty} \lambda_2(\Omega_n)=\left(\frac{|B|}{|B_R|}\right)^\frac{s\,p}{N}\, \lambda_1(B)=\lambda_1(B_R).
\]
\end{HKS}
Observe that as a consequence of this result, we obtain that the shape optimization problem
\[
\inf\{\lambda_2(\Omega)\, :\, |\Omega|= c\},
\]
{\it does not admit a solution}.

\subsection{Some words about regularity}

Our proofs are based on variational techniques for which membership of eigenfunctions to the relevant Sobolev space is sufficient.
\par
The only place where regularity of eigenfunctions is really needed in our paper is in the proof of the nonlocal Hong-Krahn-Szego inequality, which is based on the crucial Lemma \ref{lm:LL}. There continuity is necessary in order to assure that nodal domains of eigenfunctions are open sets, thus we enclose in this paper the (long) Section \ref{sec:3} where the regularity issue is tackled. There we prove at first some local and global $L^\infty$ estimates for solutions of general nonhomogeneous equations like
\begin{equation}
\label{puntualeintro}
(-\Delta_p)^s u = F,\quad \mbox{ in }\Omega,\qquad u = 0\quad \mbox{ in }\mathbb{R}^N \setminus \Omega.
\end{equation}
Though not completely new\footnote{In \cite[Theorem 1.1]{DKP} the local $L^\infty$ estimate is proved for solutions of the homogeneous equation, while \cite[Theorem 3.2]{FP} contains the global $L^\infty$ bound for eigenfunctions. Both these results use a suitable variant of De Giorgi's technique.}, these estimates are here obtained by means of a {\it Moser's iteration technique}.  For this reason, we believe them to be interesting and we prefer to include them here. We recall that in the nonlocal setting, Moser's iteration has already been employed in the linear case by Kassmann, in order to prove H\"older continuity for bounded solutions, see \cite{Ka}.   
\par
Then we show how continuity follows from the very recent result \cite[Theorem 1.5]{KMS} by Kuusi, Mingione and Sire. This point needs a precision: the regularity estimates of \cite{KMS} are indeed very general. In particular, the authors consider more general nonlinear and nonlocal operators and cover the harder case of $F$ being just a measure. In this case solutions have to be understood in a suitable very weak sense (see \cite[Definition 2]{KMS}). However, such a general setting needs the hypothesis \(p>2-s/N\),
thus their result can not be directly applied to our situation where $1<p<\infty$ and $0<s<1$ without any further restriction. Such a restriction on $p$ and $s$ in \cite{KMS} comes from a couple of crucial comparison results (essentially \cite[Lemma 3.4]{KMS}). These results can be proved in an easier way when $F$ is in the correct Lebesgue space, without any additional condition on $p$ and $s$. The scope of the second part of Section \ref{sec:3} is exactly that of showing how to fix this technical point. Then we briefly sketch the main ideas of the proof of \cite[Theorem 1.5]{KMS} for the reader's convenience.

\subsection{Plan of the paper} All the definitions, notations and preliminary results needed for the sequel are contained in Section \ref{sec:2}. In Section \ref{sec:3} we prove some regularity estimates for eigenfunctions and more generally for solutions of \eqref{puntualeintro}. The second eigenvalue of the fractional $p-$Laplacian is then introduced and studied in Section \ref{sec:4}, while Section \ref{sec:5} contains its mountain pass characterization. Finally, the nonlocal Hong-Krahn-Szego inequality is proved in Section \ref{sec:6}. The paper ends with a couple of Appendices containing some pointwise inequalities needed throughout the whole paper.  

\begin{ack}
Part of this work has been conducted during the conference  ``{\it Journ\'ees d'Analyse Appliqu\'ee Nice-Toulon-Marseille}'' held in Porquerolles in May 2014. Organizers and hosting institutions are gratefully acknowledged.
\end{ack}

\section{Definitions and preliminary results}
\label{sec:2}

\subsection{Notation} Throughout the whole paper, we will denote by $B_R(x)$ the $N-$dimensional ball having radius $R$ and center $x$. When the center will be clear from the context or unncessary, we will simply write $B_R$. Finally, $\omega_N$ is the measure of the $N-$dimensional ball with unit radius.
For a Borel set $E\subset\mathbb{R}^N$, we will denote by $|E|$ its $N-$dimensional Lebesgue measure. The average over $B_{R}(x_0)$ of a measurable function $\psi$ will be denoted by
\begin{equation}
\label{media}
\overline \psi_{x_0,R}:=\fint_{B_{R}(x_0)}\psi\, dx.
\end{equation}
We also set 
\[
\psi_+(x)=\max\{\psi(x),0\}\qquad \mbox{ and }\qquad \psi_-(x)=\max\{-\psi(x),0\},
\]
so that $\psi=\psi_+-\psi_-$.
Given $1<p<\infty$ and $0<s<1$ such that $s\,p<N$, we define
\[
p^*=\frac{N\,p}{N-s\,p}\qquad \mbox{ and }\qquad (p^*)'=\frac{N\,p}{N\,p-N+s\,p}.
\]
\subsection{Sobolev spaces} Let $0<s<1$ and $1<p<\infty$. For every $\Omega\subset\mathbb{R}^N$ open and bounded set, we consider the Sobolev space $\widetilde W^{s,p}_0(\Omega)$ defined as the completion of $C^\infty_0(\Omega)$ with respect to the norm
\[
\|u\|_{\widetilde W^{s,p}_0(\Omega)}:=\left(\int_{\mathbb{R}^N}\int_{\mathbb{R}^N} \frac{|u(x)-u(y)|^p}{|x-y|^{N+s\,p}}\, dx\,dy\right)^\frac{1}{p},\qquad u\in C^\infty_0(\Omega).
\]
We recall that the space $\widetilde W^{s,p}_0(\Omega)$ can be equivalently defined by taking the completion of $C^\infty_0(\Omega)$ with respect to the full norm
\[
\left(\int_\Omega |u|^p\, dx\right)^\frac{1}{p}+\left(\int_{\mathbb{R}^N}\int_{\mathbb{R}^N} \frac{|u(x)-u(y)|^p}{|x-y|^{N+s\,p}}\, dx\,dy\right)^\frac{1}{p},
\]
see \cite[Remark 2.5]{BLP}. If $s\,p\not =1$ and $\partial\Omega$ is smooth enough, such a space coincides with the usual one $W_0^{s,p}(\Omega)$, defined as 
the completion of $C^\infty_0(\Omega)$ with respect to the norm
\[
\left(\int_\Omega |u|^p\, dx\right)^\frac{1}{p}+\left(\int_{\Omega}\int_{\Omega} \frac{|u(x)-u(y)|^p}{|x-y|^{N+s\,p}}\, dx\,dy\right)^\frac{1}{p}.
\]
This is proven for example in \cite[Proposition B.1]{BLP}.
\begin{oss}[Borderline case $s\,p=1$]
\label{oss:sp1}
When $s\,p = 1$, the previous identification is no longer true and one has the strict inclusion
\[ 
\widetilde W^{s,p}_0(\Omega) \subsetneqq W^{s,p}_0(\Omega).
\]
Indeed, for $s\,p\le 1$ the characteristic function $1_\Omega$ of $\Omega$ can be approximated in norm by a sequence $\{u_n\}_{n\in\mathbb{N}}\subset C^\infty_0(\Omega)$, that is we have
\[
\lim_{n\to\infty} \int_{\Omega} |u_n-1_\Omega|^p\, dx=0\qquad \mbox{ and} \qquad \lim_{n\to\infty} \int_{\Omega} \int_{\Omega} \frac{|u_n(x)-u_n(y)|^p}{|x-y|^{N+s\,p}}\, dx\,dy=0,
\]
see for example the counterexample contained in \cite{Dy} at page 557.
Thus we have $1_\Omega\in W^{s,p}_0(\Omega)$. On the other hand, for $s\,p=1$ it is easily seen that
\[
\int_{\mathbb{R}^N}\int_{\mathbb{R}^N} \frac{|1_\Omega(x)-1_\Omega(y)|^p}{|x-y|^{N+1}}\, dx\,dy=2\,\int_{\Omega}\int_{\mathbb{R}^N\setminus\Omega} \frac{1}{|x-y|^{N+1}}\, dx\,dy,
\]
and the last integral does not converge, thus $1_\Omega\not\in \widetilde W^{s,p}_0(\Omega)$.
\end{oss}
The main properties of the space $\widetilde W^{s,p}_0(\Omega)$ can be found in \cite[Section 2]{BLP}. Here we state and prove a couple of simple functional inequalities that will be needed in the sequel. We use the notation
\[
W^{s,p}(\Omega):=\left\{u\in L^p(\Omega)\, :\, \int_{\Omega}\int_{\Omega} \frac{|u(x)-u(y)|^p}{|x-y|^{N+s\,p}}\, dx\,dy<+\infty\right\}.
\]
\begin{prop}[Poincar\'e with localized seminorm]
Let $1<p<\infty$ and $0<s<1$. We fix $R>0$,
for every $u\in W^{s,p}(B_R)$  we have
\begin{equation}
\label{pancarré}
\frac{|\{x\in B_R\, :\, u=0\}|}{2^{N+p}\,R^N}\,R^{-s\,p}\,\int_{B_R} |u|^p\, dx\le  \int_{B_R}\int_{B_R} \frac{|u(x)-u(y)|^p}{|x-y|^{N+s\,p}}\, dx\,dy.
\end{equation}
In particular, for $u\in\widetilde W^{s,p}_0(B_r)$ with $0<r<R$ there holds
\begin{equation}
\label{pancarré2}
\frac{N\,\omega_N}{2^{N+p}}\, \left(\frac{r}{R}\right)^{N}\,\left(\frac{R-r}{r}\right)\, R^{-s\,p}\,\int_{B_r} |u|^p\, dx\le \int_{B_{R}} \int_{B_{R}} \frac{|u(x)-u(y)|^p}{|x-y|^{N+s\,p}}\, dx\,dy.
\end{equation}
\end{prop}
\begin{proof}
Let $x\in B_R$ and we pick $y\in B_R$ such that $u(y)=0$. Then we get
\[
\begin{split}
|u(x)|^p=|u(x)-u(y)|^p&=\frac{|u(x)-u(y)|^{p}}{|x-y|^{N+s\,p}}\, |x-y|^{N+s\,p}\le 2^{N+p}\,R^{N+s\,p}\,\frac{|u(x)-u(y)|^{p}}{|x-y|^{N+s\,p}},
\end{split}
\]
where we also used that $2^{N+s\,p}<2^{N+p}$.
By integrating at first with respect to $y$ and then integrating with respect to $x$, we can get inequality \eqref{pancarré}.
\vskip.2cm\noindent
In order to prove \eqref{pancarré2}, it is sufficient to use \eqref{pancarré} and observe that 
\[
|\{x\in B_R\, :\, u=0\}|\geq\omega_N\, (R^N-r^N).
\]
By further noticing that
\[
\int_{B_R} |u|^p\, dx=\int_{B_r} |u|^p\, dx\qquad \mbox{ and }\qquad R^N-r^N\ge N\, r^{N-1}\,(R-r), 
\]
we get the conclusion after some simple manipulations.
\end{proof}
\begin{prop}[Sobolev with localized seminorm]
\label{prop:sobolev_allargati}
Let $1<p<\infty$ and $0<s<1$ such that $s\,p<N$. We fix $0<r<R$, then for every $u\in\widetilde W^{s,p}_0(B_r)$ there holds
\begin{equation}
\label{apri!}
C\,\left(\int_{B_r} |u|^{p^*}\, dx\right)^\frac{p}{p^*}\le \int_{B_{R}} \int_{B_{R}} \frac{|u(x)-u(y)|^p}{|x-y|^{N+s\,p}}\, dx\,dy,
\end{equation}
where the constant $C=C(N,p,s,R/r)>0$ goes to $0$ as $R/r$ converges to $1$. 
\end{prop}
\begin{proof}
Since $\widetilde W^{s,p}_0(B_r)\hookrightarrow W^{s,p}_0(\mathbb{R}^N)$, by Sobolev inequality in $\mathbb{R}^N$ (see for example \cite[Theorem 1]{MS}) we have
\begin{equation}
\label{sobolevtotale}
\frac{1}{T_{p,s}}\,\left(\int_{B_r} |u|^{p^*}\, dx\right)^\frac{p}{p^*}\le \int_{\mathbb{R}^N} \int_{\mathbb{R}^N} \frac{|u(x)-u(y)|^p}{|x-y|^{N+s\,p}}\, dx\,dy,
\end{equation}
where we set
\begin{equation}
\label{sharpMS}
T_{p,s}:=\sup_{v\in W^{s,p}_0(\mathbb{R}^N)} \left\{\left(\int_{\mathbb{R}^N} |v|^{p^*}\,dx\right)^\frac{p}{p^*}\, :\, \int_{\mathbb{R}^N} \int_{\mathbb{R}^N} \frac{|v(x)-v(y)|^p}{|x-y|^{N+s\,p}}\, dx\,dy=1\right\}<+\infty.
\end{equation}
We now decompose the Gagliardo seminorm as
\[
\begin{split}
\int_{\mathbb{R}^N} \int_{\mathbb{R}^N} \frac{|u(x)-u(y)|^p}{|x-y|^{N+s\,p}}\, dx\,dy&=\int_{B_{R}} \int_{B_R} \frac{|u(x)-u(y)|^p}{|x-y|^{N+s\,p}}\, dx\,dy+2\,\int_{B_R} \int_{\mathbb{R}^N\setminus B_R} \frac{|u(x)|^p}{|x-y|^{N+s\,p}}\, dx\,dy.
\end{split}
\]
By observing that $u\equiv 0$ outside $B_r$ and that for every $x\in B_r$ we have $B_{R-r}(x)\subset B_R$, the last integral is easily estimated as follows
\[
\begin{split}
\int_{B_R} \int_{\mathbb{R}^N\setminus B_R} &\frac{|u(x)|^p}{|x-y|^{N+s\,p}}\, dx\,dy=\int_{B_r} \int_{\mathbb{R}^N\setminus B_R} \frac{|u(x)|^p}{|x-y|^{N+s\,p}}\, dx\,dy\\
&\le \int_{B_r} |u|^p \,dx\,\left(\int_{\mathbb{R}^N\setminus B_{R-r}(x_0)} \frac{1}{|y-x_0|^{N+s\,p}}\, dy\right)\\
&=\frac{N\,\omega_N}{s\,p\,(R-r)^{s\,p}}\, \int_{B_r} |u|^p \,dx\\
&\le c_{N,p}\, \left(\frac{R}{r}\right)^{N}\,\left(\frac{r}{R-r}\right)\, \left(\frac{R}{R-r}\right)^{p}\,\int_{B_{R}} \int_{B_R} \frac{|u(x)-u(y)|^p}{|x-y|^{N+s\,p}}\, dx\,dy.
\end{split}
\]
In the last passage we used \eqref{pancarré2} and the fact that $R/(R-r)\ge 1$.
By using the previous in \eqref{sobolevtotale}, we get the conclusion.
\end{proof} 
\begin{oss}[Watch out!]
We point out that for $s\,p\le 1$ the previous inequalities \eqref{pancarré2} and \eqref{apri!} fail to be true for $R=r$. The counterexample is the same as in Remark \ref{oss:sp1}, i.e. the characteristic function of $B_R$. 
\end{oss}

Finally, we give a Sobolev inequality without loss on the left-hand side.
\begin{prop}[Sobolev with localized full norm]
Let $1<p<\infty$ and $0<s<1$ such that $s\,p<N$. For every $u\in W^{s,p}(B_R)$ there holds
\begin{equation}
\label{sobolev}
\|u\|^p_{L^{p^*}(B_R)}\le C\, \left[\int_{B_R} \int_{B_R} \frac{|u(x)-u(y)|^p}{|x-y|^{N+s\,p}}\, dx\,dy+\frac{1}{R^{s\,p}}\,\int_{B_R} |u|^p\, dx\right],
\end{equation}
for some constant $C=C(N,s,p)>0$.
\end{prop}
\begin{proof}
For $R=1$ the result is contained for example in \cite[Theorem 6.7]{DPV}. The case of a general $R>0$ can then be obtained by a standard scaling argument.
\end{proof}

\subsection{Nonlocal eigenvalues}
We say that $u \in \widetilde W^{s,p}_0(\Omega) \setminus \{0\}$ is an \emph{$(s,p)-$eigenfunction} associated to the \emph{eigenvalue} $\lambda$ if $u$ satisfies \eqref{sfeigenvalueintro} weakly, i.e. 
\begin{equation} 
\label{wfeigenvalue} 
\int_{\mathbb{R}^N}\int_{\mathbb{R}^N} \frac{|u(x)-u(y)|^{p-2} \,(u(x)-u(y))}{|x-y|^{N+s\,p}}\,(\varphi(x)-\varphi(y))\, dx\,dy = \lambda \int_\Omega |u|^{p-2}\, u\,\varphi\,dx,
\end{equation}
for every $\varphi \in \widetilde W^{s,p}_0(\Omega)$.  
If we set 
\begin{equation}
\label{sfera}
\mathcal{S}_p(\Omega)=\left\{u\in \widetilde W^{s,p}_0(\Omega)\, :\, \int_\Omega |u|^p\, dx=1\right\},
\end{equation}
we already observed in the Introduction that $(s,p)-$eigenvalues coincide with critical points of the functional
\begin{equation}
\label{phi}
\Phi_{s,p}(u)=\|u\|^p_{\widetilde W^{s,p}_0(\Omega)}=\int_{\mathbb{R}^N}\int_{\mathbb{R}^N} \frac{|u(x)-u(y)|^p}{|x-y|^{N+s\,p}}\, dx\,dy,
\end{equation}
restricted to the manifold $\mathcal{S}_p(\Omega)$. The first $(s,p)-$eigenvalue of $\Omega$ is given by
\begin{equation}
\label{problema}
\lambda_1(\Omega)=\min_{u\in \mathcal{S}_p(\Omega)} \|u\|^p_{\widetilde W^{s,p}_0(\Omega)}.
\end{equation}
It is easy to see that first eigenfunctions must be nonnegative (or nonpositive). This is a consequence of the elementary inequality
\[ 
\big||u(x)|-|u(y)|\big|\leq |u(x)-u(y)|,
\]
which holds strictly whenever $u(x)\,u(y)<0$.
\vskip.2cm
It can be proved actually that any constant sign $(s,p)-$eigenfunction must be strictly positive (or strictly negative) on every open bounded set $\Omega$, {\it even disconnected}. This is the content of the following result, which holds true without any connectedness assumptions and which therefore extends \cite[Theorem A.1]{BF}.
\begin{prop}[Minimum principle] 
\label{lm:lemmamagico}
Let $1<p<\infty$ and $0<s<1$. Let $u \in \widetilde W^{s,p}_0(\Omega) \setminus \{0\}$ be a nonnegative $(s,p)-$eigenfunction associated to the eigenvalue $\lambda$. Then $u > 0$ in $\Omega$.
\end{prop}
\begin{proof}
If $\Omega$ is connected, this is exactly \cite[Theorem A.1]{BF}. Let us suppose that $\Omega$ is not connected.
Still by \cite[Theorem A.1]{BF} we already know that $u>0$ on each connected component of $\Omega$ where $u$ is not identically zero. Set $\Omega_1 := \{ x \in \Omega\,:\,u(x)>0\}$, and suppose by contradiction that there exists a connected component $\Omega_2$ of $\Omega$ such that $u$ vanishes almost everywhere on $\Omega_2$. Let $\varphi \in C^\infty_0(\Omega_2)$ be any nonnegative test function, not identically zero. By inserting this in \eqref{wfeigenvalue}, with simple manipulations we get
\[
\begin{split} 
0 = \lambda \int_\Omega u^{p-1}\,\varphi\,dx &= \int_{\mathbb{R}^N}\int_{\mathbb{R}^N} \frac{|u(x)-u(y)|^{p-2} \,(u(x)-u(y))}{|x-y|^{N+s\,p}}\,(\varphi(x)-\varphi(y))\, dx\,dy \\ 
& = -2\,\int_{\Omega_1}\int_{\Omega_2} \frac{u(x)^{p-1}}{|x-y|^{N+s\,p}}\,\varphi(y)\, dx\,dy.
\end{split}
\]
Therefore $u \equiv 0$ in $\Omega_1$, a contradiction. Hence $u>0$ in $\Omega$.
\end{proof}
\begin{oss}
The previous result is in contrast with the local case, where actually nonnegative eigenfunctions can identically vanish on some connected components.
\end{oss}
The following statement summarizes some basic facts about the first eigenvalue.
\begin{teo}
\label{teo:primo!}
Let $1<p<\infty$ and $0<s<1$. For every $\Omega\subset\mathbb{R}^N$ open and bounded set we have:
\begin{enumerate} 
\item[{\it i)}] any first $(s,p)-$eigenfunction must be strictly positive (or strictly negative);
\vskip.2cm
\item[{\it ii)}] $\lambda_1(\Omega)$ is simple, i.e. the solution of \eqref{problema} is unique, up to the choice of the sign;
\vskip.2cm
\item[{\it iii)}] if $u$ is an eigenfunction associated to an eigenvalue $\lambda > \lambda_1(\Omega)$, then $u$ must be sign-changing.
\end{enumerate}
\end{teo}
\begin{proof}
The first item follows from Proposition \ref{lm:lemmamagico}. 
\par
Simplicity of $\lambda_1(\Omega)$ follows from \cite[Theorem 4.2]{FP}, which holds true also for disconnected domains, thanks to Proposition \ref{lm:lemmamagico}. 
\par
The last statement for general sets follows combining Proposition \ref{lm:lemmamagico} and \cite[Theorem 4.1]{FP} (again, this last result is valid also without assuming that $\Omega$ is connected).
\end{proof}
\begin{oss}
As in the local case, the exact structure of the spectrum $\sigma_{s,p}(\Omega)$ is an open issue. However, it is possible to show that it certainly contains an increasing sequence of eigenvalues $\{\lambda_k(\Omega)\}_{k\in\mathbb{N}}$ diverging at infinity, by means of standard minimax procedures (see for example \cite[Proposition 2.2]{IS}). We also mention the interesting result \cite[Theorem 1.1]{IS} about the asymptotic distribution of $(s,p)-$eigenvalues.
\end{oss}

\section{A journey into regularity for the fractional $p-$Laplacian}
\label{sec:3}

The aim of this section is to prove that an $(s,p)-$eigenfunction is continuous, for every $1<p<\infty$ and $0<s<1$. Indeed, we will consider the general case of solutions to \begin{equation}
\label{puntuale}
(-\Delta_p)^s u = F,\quad \mbox{ in }\Omega,\qquad u = 0\quad \mbox{ in }\mathbb{R}^N \setminus \Omega,
\end{equation}
where $F\in L^{(p^*)'}(\Omega)$ is given. We will prove some global and local $L^\infty$ estimates and then explain how continuity follows from the recent result by Kuusi, Mingione and Sire contained in \cite{KMS}, concerning the case of $F$ being just a measure (not necessarily belonging to the dual of $\widetilde W^{s,p}_0(\Omega)$). 
\par
The reader not interested in regularity issues is invited to skip this (long) section and go directly to Section \ref{sec:4}.
\vskip.2cm
We point out that for $s\,p>N$ we already know that $\widetilde W^{s,p}_0(\Omega)\hookrightarrow L^\infty(\Omega)\cap C^{0,s-N/p}(\Omega)$ (see \cite[Proposition 2.9]{BLP}). Then we only need to consider the case $s\,p\le N$.
For simplicity, we will consider the case $s\,p< N$, then it will be evident how to handle the borderline case $s\,p=N$, where $F\in L^q$ for $q>1$. In this last case, it will be sufficient to reproduce the proofs of this section, by replacing the continuous embedding $W^{s,p}\hookrightarrow L^{p^*}$ (and the associated Sobolev inequality) with  $W^{s,N/s}\hookrightarrow L^{m}$ and $m>1$ large enough.
\par
Throughout the whole section, given $F\in L^{(p^*)'}(\Omega)$ we always denote by $u\in \widetilde W^{s,p}_0(\Omega)$ the solution to \eqref{puntuale},
i.e. $u$ satisfies
\begin{equation}
\label{eigeneq}
\begin{split}
\int_{\mathbb{R}^N}\int_{\mathbb{R}^N} &\frac{|u(x)-u(y)|^{p-2}\, (u(x)-u(y))}{|x-y|^{N+s\,p}}\, (\varphi(x)-\varphi(y))\, dx\, dy=\int_\Omega F\, \varphi\, dx,
\end{split}
\end{equation}
for every $\varphi\in \widetilde W^{s,p}_0(\Omega)$. 

\subsection{Global boundedness}

We start with the following global result. 
\begin{teo}[Global $L^\infty$ bound]
\label{teo:linfty}
Let $1<p<\infty$ and $0<s<1$ be such that $s\,p< N$. If $F\in L^q(\Omega)$ for $q>N/(s\,p)$, then $u\in L^\infty(\Omega)$. Moreover, we have the scaling invariant estimate
\begin{equation}
\label{stimonaLinfty}
\|u\|_{L^\infty(\Omega)}\le \left(C\,\chi^\frac{1}{\chi-1}\right)^\frac{\chi}{\chi-1}\,\left(T_{p,s}\,|\Omega|^{\frac{s\,p}{N}-\frac{1}{q}}\,\|F\|_{L^q(\Omega)}\right)^\frac{1}{p-1},
\end{equation}
where $C=C(p)>0$, $T_{p,s}$ is the sharp Sobolev constant defined in \eqref{sharpMS} and $\chi=p^*/(p\,q')$.
\end{teo}
\begin{proof}
We assume at first $p\ge 2$. For every $0<\varepsilon\ll 1$, we define the smooth convex Lipschitz function
\[
f_\varepsilon(t)=(\varepsilon^2+t^2)^\frac{1}{2},
\]
then we insert the test function $\varphi=\psi\, |f'_\varepsilon(u)|^{p-2}\, f'_\varepsilon(u)$ in \eqref{eigeneq}, where $\psi\in C^\infty_0(\Omega)$ is a positive function. By using \eqref{tassello} with the choices
\[
\tau=0,\qquad a=u(x),\qquad b=u(y),\qquad A=\psi(x) \qquad \mbox{ and }\qquad B=\psi(y),
\]
we get 
\[
\begin{split}
\int_{\mathbb{R}^N}\int_{\mathbb{R}^N} \frac{|f_\varepsilon(u(x))-f_\varepsilon(u(y))|^{p-2}\, (f_\varepsilon(u(x))-f_\varepsilon(u(y)))}{|x-y|^{N+s\,p}}\, &(\psi(x)-\psi(y))\, dx\, dy\\
&\le  \int_\Omega |F|\, |f'_\varepsilon(u)|^{p-1}\, \psi\, dx,
\end{split}
\]
By observing that $f_\varepsilon$ converges to $f(t)=|t|$, using that $|f_\varepsilon'(t)|\le 1$ and Fatou Lemma\footnote{Observe that the integrand on the left-hand side can be estimated from below by the integrable function
\[
-\frac{p-1}{p}\,\frac{\Big||u(x)|-|u(y)|\Big|^{p}}{|x-y|^{N+s\,p}}-\frac{1}{p}\,\frac{|\psi(x)-\psi(y)|}{|x-y|^{N+s\,p}},
\]
thanks to Young inequality and to the $1-$Lipschitz character of $f_\varepsilon$.}, if we pass to the limit in the previous we get that the function $|u|$ verifies
\begin{equation}
\label{eigeneqmod}
\int_{\mathbb{R}^N}\int_{\mathbb{R}^N} \frac{\Big||u(x)|-|u(y)|\Big|^{p-2}\, \Big(|u(x)|-|u(y)|\Big)}{|x-y|^{N+s\,p}}\, (\psi(x)-\psi(y))\, dx\, dy\le \int_\Omega |F|\, \psi\, dx,
\end{equation}
for every positive function $\psi\in C^\infty_0(\Omega)$. By density the same remains true for $\psi\in \widetilde W^{s,p}_0(\Omega)$, still with $\psi\ge 0$.  
\par
The same equation still holds true for $1<p<2$. The test function above is no more a legitimate one and we need to slightly modify it as follows
\[
\varphi=\psi\, \left(\varepsilon+|f'_\varepsilon(u)|^2\right)^\frac{p-2}{2}\, f'_\varepsilon(u).
\]
By using \eqref{tassello}, this time in its full generality (i.e. with $\tau=\varepsilon>0$)
\[
\begin{split}
\int_{\mathbb{R}^N}\int_{\mathbb{R}^N}& \frac{\Big(\varepsilon\,(u(x)-u(y))^2+(f_\varepsilon(u(x))-f_\varepsilon(u(y)))^2\Big)^\frac{p-2}{2}\,(f_\varepsilon(u(x))-f_\varepsilon(u(y)))\, (\psi(x)-\psi(y))}{|x-y|^{N+s\,p}}\,dx\,dy\\
&\le \int_\Omega |F|\, |f'_\varepsilon(u)|^{p-1}\, \psi\, dx.
\end{split}
\]
We now observe that on the left-hand side we can pass to the limit as $\varepsilon$ goes to $0$, since for $1<p<2$ we have
\[
\begin{split}
&\frac{\Big(\varepsilon\,(u(x)-u(y))^2+(f_\varepsilon(u(x))-f_\varepsilon(u(y)))^2\Big)^\frac{p-2}{2}\,|f_\varepsilon(u(x))-f_\varepsilon(u(y))|\, |\psi(x)-\psi(y)|}{|x-y|^{N+s\,p}}\\
&\le \frac{|f_\varepsilon(u(x))-f_\varepsilon(u(y))|^{p-1}\, |\psi(x)-\psi(y)|}{|x-y|^{N+s\,p}}\\
&\le \frac{|u(x)-u(y)|^{p-1}\, |\psi(x)-\psi(y)|}{|x-y|^{N+s\,p}}\in L^1(\mathbb{R}^N\times\mathbb{R}^N).
\end{split}
\]
In conclusion, we get \eqref{eigeneqmod} for $1<p<2$ as well.
\par
For every $M>0$, we now define $u_M=\min\{|u|,M\}$ and observe that $u_M$ is still in $\widetilde W^{s,p}_0(\Omega)$, since this is just the composition of $u$ with a Lipschitz function vanishing in $0$.
Given $\beta>0$ and $\delta>0$, we insert the test function $\psi=(u_M+\delta)^\beta-\delta^\beta$ in \eqref{eigeneqmod}, then we get
\[
\begin{split}
\int_{\mathbb{R}^N}\int_{\mathbb{R}^N}& \frac{\Big||u(x)|-|u(y)|\Big|^{p-2}\, (|u(x)|-|u(y)|)\, \Big((u_M(x)+\delta)^\beta-(u_M(y)+\delta)^\beta\Big)}{|x-y|^{N+s\, p}}\, dx\, dy\\
&\le \int_{\Omega} |F|\,(u_M+\delta)^{\beta}\, dx.
\end{split}
\]
By using inequality \eqref{tassello2} with the function
\[
g(t)=(\min\{t,M\}+\delta)^\beta,
\]
from the previous we get
\[
\begin{split}
\frac{\beta\, p^p}{(\beta+p-1)^p}\, \int_{\mathbb{R}^N}\int_{\mathbb{R}^N} &\frac{\left|(u_M(x)+\delta)^\frac{\beta+p-1}{p}-(u_M(y)+\delta)^\frac{\beta+p-1}{p}\right|^p}{|x-y|^{N+s\, p}}\, dx\,dy\le \int_{\Omega} |F|\, (u_M+\delta)^\beta\, dx.
\end{split}
\]
We can now use the Sobolev inequality for $W^{s,p}_0(\mathbb{R}^N)$ (see \cite[Theorem 1]{MS}), so to get
\[
\begin{split}
\int_{\mathbb{R}^N}\int_{\mathbb{R}^N} &\frac{\left|(u_M(x)+\delta)^\frac{\beta+p-1}{p}-(u_M(y)+\delta)^\frac{\beta+p-1}{p}\right|^p}{|x-y|^{N+s\, p}}\, dx\,dy\\
&\ge \frac{1}{T_{p,s}}\,\left\|(u_M+\delta)^\frac{\beta+p-1}{p}-\delta^\frac{\beta+p-1}{p}\right\|^p_{L^{p^*}(\Omega)}.
\end{split}
\]
We then obtain
\[
\begin{split}
\left\|(u_M+\delta)^\frac{\beta+p-1}{p}-\delta^\frac{\beta+p-1}{p}\right\|^p_{L^{p^*}(\Omega)}&\le T_{p,s}\,\frac{\|F\|_{L^q(\Omega)}}{\beta}\, \left(\frac{\beta+p-1}{p}\right)^{p}\, \left\|(u_M+\delta)^{\beta}\right\|_{L^{q'}(\Omega)}.
\end{split}
\]
We also observe that by triangle inequality and the simple inequality $(u_M+\delta)^{\beta+p-1}\ge \delta^{p-1}\, (u_M+\delta)^{\beta}$, the left-hand side can be estimated by
\[
\begin{split}
\left\|(u_M+\delta)^\frac{\beta+p-1}{p}-\delta^\frac{\beta+p-1}{p}\right\|^p_{L^{p^*}(\Omega)}&\ge \left(\frac{\delta}{2}\right)^{p-1}\,\left\|(u_M+\delta)^\frac{\beta}{p}\right\|^p_{L^{p^*}(\Omega)}-\delta^{\beta+p-1}\, |\Omega|^\frac{N-s\,p}{N}.
\end{split}
\]
By using this estimate, up to now we gained
\[
\begin{split}
\left\|(u_M+\delta)^\frac{\beta}{p}\right\|^p_{L^{p^*}(\Omega)}&\le C\,T_{p,s}\,\frac{\|F\|_{L^q(\Omega)}}{\beta}\, \left(\frac{\beta+p-1}{p\,\delta^\frac{p-1}{p}}\right)^{p}\, \, \left\|(u_M+\delta)^{\beta}\right\|_{L^{q'}(\Omega)}+C\,\delta^{\beta}\, |\Omega|^\frac{N-s\,p}{N},
\end{split}
\]
for a constant $C=C(p)>0$. On the other hand, it is easy to see that\footnote{We use that 
\[
\frac{1}{\beta}\,\left(\frac{\beta+p-1}{p}\right)^p \ge 1,
\]
see Lemma \ref{betaandpi}.}
\[
\delta^{\beta}\, |\Omega|^\frac{N-s\,p}{N}\le \frac{1}{\beta}\,\left(\frac{\beta+p-1}{p}\right)^p\,|\Omega|^{1-\frac{1}{q'}-\frac{s\,p}{N}}\, \left\|(u_M+\delta)^{\beta}\right\|_{L^{q'}(\Omega)}.
\]
Thus we get
\begin{equation}
\label{invertito}
\begin{split}
\left\|(u_M+\delta)^\frac{\beta}{p}\right\|^p_{L^{p^*}(\Omega)}&\le C\,\frac{1}{\beta}\,\left(\frac{\beta+p-1}{p}\right)^p\,\left\|(u_M+\delta)^{\beta}\right\|_{L^{q'}(\Omega)}\, \left[\frac{T_{p,s}\,\|F\|_{L^q(\Omega)}}{\delta^{p-1}}+|\Omega|^{1-\frac{1}{q'}-\frac{s\,p}{N}}\right],
\end{split}
\end{equation}
with $C=C(p)>0$. We then choose $\delta>0$ by
\begin{equation}
\label{deltone}
\delta=\left(T_{p,s}\,\|F\|_{L^q(\Omega)}\right)^\frac{1}{p-1}\,|\Omega|^{-\frac{1}{p-1}\,(1-\frac{1}{q'}-\frac{s\,p}{N})},
\end{equation}
and we restrict to $\beta\ge 1$, so that
\[
\frac{1}{\beta}\,\left(\frac{\beta+p-1}{p}\right)^p\le \beta^{p-1}.
\] 
If we further introduce $\vartheta=\beta\,q'$, then the previous inequality can be written as
\[
\|u_M+\delta\|_{L^{\chi\,\vartheta}(\Omega)}\le \left[C\,|\Omega|^{1-\frac{1}{q'}-\frac{s\,p}{N}}\right]^\frac{q'}{\vartheta}\,\left[\left(\frac{\vartheta}{q'}\right)^\frac{q'}{\vartheta}\right]^{p-1}\, \|u_M+\delta\|_{L^{\vartheta}(\Omega)},
\]
where as in the statement, we set
\begin{equation}
\label{chi}
\chi:=\frac{p^*}{p\,q'}=\frac{N}{N-s\,p}\,\frac{1}{q'}>1.
\end{equation}
We now iterate the previous inequality, by taking the following sequence of exponents
\[
\vartheta_0=q',\qquad\qquad \vartheta_{n+1}=\chi\,\vartheta_n=\chi^{n+1}\,q'.
\]
Since $\chi>1$, then 
\[
\sum_{n=0}^\infty \frac{q'}{\vartheta_n}=\sum_{n=0}^\infty \frac{1}{\chi^n}=\frac{\chi}{\chi-1}=\frac{N}{N-N\,q'+s\,p\,q'},
\]
and
\[
\prod_{n=0}^\infty \left(\frac{\vartheta_n}{q'}\right)^\frac{{q'}}{\vartheta_n}=\chi^\frac{\chi}{(\chi-1)^2}.
\]
By starting from $n=0$, at the step $n$ we have
\[
\big\|u_M+\delta\big\|_{L^{\vartheta_{n+1}}(\Omega)}\le \left[C\,|\Omega|^{1-\frac{1}{q'}-\frac{s\,p}{N}}\right]^{\sum\limits_{i=0}^n\frac{q'}{\vartheta_i}}\,\left[\prod_{i=0}^n \left(\frac{\vartheta_i}{q'}\right)^\frac{{q'}}{\vartheta_i}\right]^{p-1} \big\|u_M+\delta\big\|_{L^{q'}(\Omega)}.
\]
By taking the limit as $n$ goes to $\infty$ we finally obtain
\[
\|u_M\|_{L^\infty(\Omega)}\le \left(C\,\chi^\frac{1}{\chi-1}\right)^\frac{\chi}{\chi-1}\,\left(|\Omega|^{1-\frac{1}{q'}-\frac{s\,p}{N}}\right)^\frac{\chi}{\chi-1}\, \|u_M+\delta\|_{L^{q'}(\Omega)},
\]
for some constant $C=C(p)>0$. In particular, since $u_M\le |u|$ and by triangle inequality, we get
\[
\|u_M\|_{L^\infty(\Omega)}\le \left(C\,\chi^\frac{1}{\chi-1}\right)^\frac{\chi}{\chi-1}\,\left(|\Omega|^{1-\frac{1}{q'}-\frac{s\,p}{N}}\right)^{\frac{\chi}{\chi-1}}\, \Big[\|u\|_{L^{q'}(\Omega)}+\delta\, |\Omega|^\frac{1}{q'}\Big].
\] 
By recalling \eqref{deltone} and \eqref{chi}, if we let $M$ go to $\infty$ we finally get
\begin{equation}
\label{stimettaLinfty}
\|u\|_{L^\infty(\Omega)}\le \left(C\,\chi^\frac{1}{\chi-1}\right)^\frac{\chi}{\chi-1}\,\left[|\Omega|^{-\frac{1}{q'}}\,\|u\|_{L^{q'}(\Omega)}+\,\left(T_{p,s}\,|\Omega|^{\frac{s\,p}{N}-\frac{1}{q}}\,\|F\|_{L^q(\Omega)}\right)^\frac{1}{p-1}\right],
\end{equation}
which shows in particular that $u\in L^\infty(\Omega)$.
\vskip.2cm\noindent
In order to get \eqref{stimonaLinfty}, we only need to estimate the $L^{q'}$ norm of $u$. We first observe that by combining H\"older and Sobolev inequalities, we get
\begin{equation}
\label{soboboing}
|\Omega|^{-\frac{1}{q'}}\,\|u\|_{L^{q'}(\Omega)}\le {T}_{p,s}^{\frac{1}{p}}\,|\Omega|^{-\frac{1}{p}+\frac{s}{N}}\,\|u\|_{\widetilde W^{s,p}_0(\Omega)}.
\end{equation}
On the other hand, by testing \eqref{eigeneq} with $u$ itself and then using H\"older and Young inequalities in conjunction with \eqref{soboboing}, we get
\[
\begin{split}
\|u\|^p_{\widetilde W^{s,p}_0(\Omega)}=\int_\Omega F\,u\,dx\le \|F\|_{L^q(\Omega)}\, \|u\|_{L^{q'}(\Omega)}&\le T_{p,s}^\frac{1}{p}\,|\Omega|^{\frac{1}{q'}-\frac{1}{p}+\frac{s}{N}}\,\|F\|_{L^q(\Omega)}\, \|u\|_{\widetilde W^{s,p}_0(\Omega)}\\
&\le \frac{1}{p'}\,\left(T_{p,s}^\frac{1}{p}\,|\Omega|^{\frac{1}{q'}-\frac{1}{p}+\frac{s}{N}}\,\|F\|_{L^q(\Omega)}\right)^{p'}+\frac{1}{p}\,\|u\|_{\widetilde W^{s,p}_0(\Omega)}^p,
\end{split}
\]
whence
\begin{equation}
\label{youngscazzata}
\|u\|_{\widetilde W^{s,p}_0(\Omega)}\le T_{p,s}^\frac{1}{p\,(p-1)}\,\left(|\Omega|^{\frac{1}{q'}-\frac{1}{p}+\frac{s}{N}}\,\|F\|_{L^q(\Omega)}\right)^\frac{1}{p-1}.
\end{equation}
By combining \eqref{soboboing} and \eqref{youngscazzata}, and using the resulting estimate in \eqref{stimettaLinfty}, we get the desired conclusion.
\end{proof}
\begin{oss}[Eigenfunctions]
\label{oss:eigenlimitate}
For eigenfunctions, starting from \eqref{eigeneqmod} one can reproduce the proof in \cite[Theorem 3.1]{BLP} for the first eigenfunction and prove that $u\in L^\infty(\Omega)$. This gives another proof of \cite[Theorem 3.2]{FP}, this time based on Moser's iterations. For $s\,p<N$ we obtain as in \cite{BLP} the estimate
\[
\|u\|_{L^\infty(\Omega)}\le \left[\widetilde C_{N,p,s}\,\lambda\right]^\frac{N}{s\,p^2}\,\|u\|_{L^p(\Omega)},
\]
where $\widetilde C_{N,p,s}$ is linked to the sharp Sobolev constant $T_{p,s}$ through (see \cite[Remark 3.4]{BLP})
\[
\widetilde C_{N,p,s}=T_{p,s}\, \left(\frac{p^*}{p}\right)^{\frac{N-s\, p}{s}\,\frac{p-1}{p}}.
\]
By Lebesgue interpolation, the previous also gives
\[
\|u\|_{L^\infty(\Omega)}\le \left[\widetilde C_{N,p,s}\,\lambda\right]^\frac{N}{s\,p}\,\|u\|_{L^1(\Omega)}.
\]
\end{oss}

\subsection{Caccioppoli inequalities}

The previous result is based on the fact that if $u$ is a solution of \eqref{eigeneq}, then $|u|$ is a subsolution of the same equation. Like in the local case, this is a general fact which remains true for every $f(u)$ with $f$ convex. This is the content of the next results.
\begin{lm}[Subsolutions, part I]
\label{lm:subsolution1}
Let $1<p<\infty$ and $0<s<1$. Let $F\in L^q(\Omega)$ for $q>N/(s\,p)$ and let $u\in \widetilde W^{s,p}_0(\Omega)\cap L^\infty(\Omega)$ be the solution of \eqref{eigeneq}. Then for every $f:\mathbb{R}\to\mathbb{R}$ convex $C^1$, the composition $v=f\circ u$ verifies
\begin{equation}
\label{subeigeneq1}
\begin{split}
\int_{\mathbb{R}^N}\int_{\mathbb{R}^N}\frac{|v(x)-v(y)|^{p-2}\, (v(x)-v(y))}{|x-y|^{N+s\,p}}\, (\varphi(x)-\varphi(y))\, dx\, dy\le \int_\Omega F\,|f'(u)|^{p-2}\,f'(u)\, \varphi\, dx,
\end{split}
\end{equation}
for every positive function $\varphi\in \widetilde W^{s,p}_0(\Omega)$.
\end{lm}
\begin{proof}
Observe that the assumption $u\in L^\infty(\Omega)$ is not restrictive, by Theorem \ref{teo:linfty}.
The proof goes exactly like in the first part of Theorem \ref{teo:linfty}. We insert in \eqref{eigeneq} the test function $\varphi=\psi\, |f'(u)|^{p-2}\, f'(u)$, where $\psi\in C^\infty_0(\Omega)$ is a positive function (for $1<p<2$ we need to modify it as before). If $f$ is not regular enough, a standard smoothing argument will be needed, we leave the details to the reader.
Then by appealing to \eqref{tassello}, we get \eqref{subeigeneq1} for smooth test functions. A density argument gives again the conclusion.
\end{proof}
The following is very similar, but we lower the hypothesis on $F$.
\begin{lm}[Subsolutions, part II]
\label{lm:subsolution2}
Let $1<p<\infty$ and $0<s<1$ be such that $s\,p<N$. Let $F\in L^{(p^*)'}(\Omega)$ and $u\in \widetilde W^{s,p}_0(\Omega)$ be the solution of \eqref{eigeneq}. Then for every $f:\mathbb{R}\to\mathbb{R}$ convex and $L-$Lipschitz, the composition $v=f\circ u$ verifies
\begin{equation}
\label{subeigeneq2}
\begin{split}
\int_{\mathbb{R}^N}\int_{\mathbb{R}^N} &\frac{|v(x)-v(y)|^{p-2}\, (v(x)-v(y))}{|x-y|^{N+s\,p}}\, (\varphi(x)-\varphi(y))\, dx\, dy\le L^{p-1}\,\int_\Omega |F|\, \varphi\, dx,
\end{split}
\end{equation}
for every positive function $\varphi\in \widetilde W^{s,p}_0(\Omega)$.
\end{lm}
The following can be seen as the Moser-type counterpart of \cite[Theorem 1.4]{DKP}, for equations with a right-hand side $F$. 
\begin{prop}[Localized Caccioppoli inequality]
\label{prop:caccioloc}
Let $1<p<\infty$ and $0<s<1$ be such that $s\,p< N$ and let $F\in L^{(p^*)'}(\Omega)$.
We assume that $v=f(u)$ is a subsolution, i.e. a function satisfying \eqref{subeigeneq2} for an $L-$Lipschitz convex function $f$ and such that
 \begin{equation}
\label{positività}
v\ge 0\qquad \mbox{ on } \Omega'\Subset\Omega.
\end{equation}
For every $\beta\ge 1$ and $\delta\ge 0$, we take
\[
g(t)=(t+\delta)^\beta\qquad \mbox{ and }\qquad G(t)=\int_0^t g'(\tau)^\frac{1}{p}\, d\tau=\frac{p\,\beta^\frac{1}{p}}{\beta+p-1}\,(t+\delta)^\frac{\beta+p-1}{p},\qquad t\ge 0.
\]
Then for every positive $\psi\in C^\infty_0(\Omega)$ such that $\mathrm{spt}(\psi)\Subset\Omega'$ we have
\begin{equation}
\label{monster}
\begin{split}
\int_{\Omega'}\int_{\Omega'} &\frac{\Big|G(v(x))\,\psi(x)-G(v(y))\,\psi(y)\Big|^{p}\,}{|x-y|^{N+s\,p}}\, dx\, dy\\
&\le \frac{C}{\beta}\,\left(\frac{\beta+p-1}{p}\right)^p\,\int_{\Omega'}\int_{\Omega'} \frac{|\psi(x)-\psi(y)|^p}{|x-y|^{N+s\,p}}\,\Big(G(v(x))^p+G(v(y))^p\Big)\, dx\,dy\\
&+C\,\left(\sup_{y\in \mathrm{spt}(\psi)} \int_{\mathbb{R}^N\setminus \Omega'} \frac{|v(x)|^{p-1}}{|x-y|^{N+s\,p}}\,dx\right)\, \int_{\Omega'} g(v)\, \psi^p\, dx\\
&+C\,L^{p-1}\,\int_\Omega |F|\,g(v)\, \psi^p\, dx,
\end{split}
\end{equation}
for some constant $C=C(p)>0$.
\end{prop}
\begin{proof}
We insert in \eqref{subeigeneq2} the test function 
$\varphi=\psi^p\,g(v)$, where $\psi\in C^\infty_0(\Omega)$ is a positive function such that $\mathrm{spt}(\psi)\Subset\Omega'$. Then we get
\begin{equation}
\label{subeigeneq20}
\begin{split}
\int_{\mathbb{R}^N}\int_{\mathbb{R}^N} &\frac{|v(x)-v(y)|^{p-2}\, (v(x)-v(y))}{|x-y|^{N+s\,p}}\, (g(v(x))\,\psi(x)^p-g(v(y))\,\psi(y)^p)\, dx\, dy\\
&\le L^{p-1}\,\int_\Omega |F|\,g(v)\, \psi^p\, dx.
\end{split}
\end{equation}
We now split the double integral in three parts:
\[
\mathcal{I}_1=\int_{\Omega'}\int_{\Omega'} \frac{|v(x)-v(y)|^{p-2}\, (v(x)-v(y))}{|x-y|^{N+s\,p}}\, (g(v(x))\,\psi(x)^p-g(v(y))\,\psi(y)^p)\, dx\, dy,
\]
\[
\mathcal{I}_2=\int_{\mathbb{R}^N\setminus\Omega'}\int_{\Omega'} \frac{|v(x)-v(y)|^{p-2}\, (v(x)-v(y))}{|x-y|^{N+s\,p}}\, g(v(x))\,\psi(x)^p\, dx\, dy,
\]
and
\[
\mathcal{I}_3=-\int_{\Omega'}\int_{\mathbb{R}^N\setminus\Omega'} \frac{|v(x)-v(y)|^{p-2}\, (v(x)-v(y))}{|x-y|^{N+s\,p}}\, g(v(y))\,\psi(y)^p\, dx\, dy
\]
{\bf Estimate of $\mathcal{I}_1$.} 
For the first integral $\mathcal{I}_1$, we proceed as follows: at first, for every $x,y\in\Omega'$ we have
\[
\begin{split}
|v(x)-v(y)|^{p-2}\, &(v(x)-v(y))\,(g(v(x))\,\psi(x)^p-g(v(y))\,\psi(y)^p)\\
&=|v(x)-v(y)|^{p-2}\,(v(x)-v(y))\,\frac{g(v(x))-g(v(y))}{2}\, \left(\psi(x)^p+\psi(y)^p\right)\\
&+|v(x)-v(y)|^{p-2}\, (v(x)-v(y))\, \frac{g(v(x))+g(v(y))}{2}\,(\psi(x)^p-\psi(y)^p).\\
&\ge |v(x)-v(y)|^{p-2}\,(v(x)-v(y))\,\frac{g(v(x))-g(v(y))}{2}\, \left(\psi(x)^p+\psi(y)^p\right)\\
&-|v(x)-v(y)|^{p-1}\, \frac{g(v(y))+g(v(x))}{2}\,|\psi(x)^p-\psi(y)^p|.
\end{split}
\]
In order to estimate the last term, we use at first
\begin{equation}
\label{quadrato}
\begin{split}
|\psi(x)^p-\psi(y)^p|&\le p\, \Big(\psi(x)^p+\psi(y)^p\Big)^\frac{p-1}{p}\,|\psi(x)-\psi(y)|.
\end{split}
\end{equation} 
Then we appeal to the definition of $g(t)$ and to Young inequality with exponents $p$ and $p'$. By using \eqref{quadrato} we have
\[
\begin{split}
|v(x)-v(y)|^{p-1}\, &\frac{(v(y)+\delta)^\beta+(v(x)+\delta)^\beta}{2}\,|\psi(x)^p-\psi(y)^p|\\
&\le \varepsilon\,\frac{p-1}{2}\, |v(x)-v(y)|^p\, (v(x)+\delta)^{\beta-1}\,(\psi(x)^p+\psi(y)^p)\\
&+ \frac{\varepsilon^{-(p-1)}}{2}\,(v(x)+\delta)^{\beta+p-1}\, |\psi(x)-\psi(y)|^p\\
&+\varepsilon\,\frac{p-1}{2}\, |v(x)-v(y)|^p\, (v(y)+\delta)^{\beta-1}\,(\psi(x)^p+\psi(y)^p)\\
&+ \frac{\varepsilon^{-(p-1)}}{2}\,(v(y)+\delta)^{\beta+p-1}\, |\psi(x)-\psi(y)|^p,
\end{split}
\]
where $\varepsilon>0$ will be chosen in a while.
For the moment we have shown
\[
\begin{split}
|v(x)-v(y)|^{p-2}\, &(v(x)-v(y))\,(g(v(x))\,\psi(x)^p-g(v(y))\,\psi(y)^p)\\
&\ge |v(x)-v(y)|^{p-2}\,(v(x)-v(y))\,\frac{g(v(x))-g(v(y))}{2}\, \left(\psi(x)^p+\psi(y)^p\right)\\
&-\varepsilon\,\frac{p-1}{2}\, |v(x)-v(y)|^p\, \Big((v(x)+\delta)^{\beta-1}+(v(y)+\delta)^{\beta-1}\Big)\,(\psi(x)^p+\psi(y)^p)\\
&- \frac{\varepsilon^{-(p-1)}}{2}\,\Big((v(x)+\delta)^{\beta+p-1}+(v(y)+\delta)^{\beta+p-1}\Big)\, |\psi(x)-\psi(y)|^p.
\end{split}
\]
Now we can use the pointwise inequality \eqref{rottolepalle} for the second term in the right-hand side, with the choices
\[
a=v(x)+\delta\qquad \mbox{ and }\qquad b=v(y)+\delta.
\]
This gives us\footnote{We use that the constant $\max\{1,3-\beta)\}$ appearing in \eqref{rottolepalle} is less than $2$.}
\[
\begin{split}
&|v(x)-v(y)|^{p-2}\,(v(x)-v(y))\,(g(v(x))\,\psi(x)^p-g(v(y))\,\psi(y)^p)\\
&\ge \frac{1-2\,\varepsilon\,(p-1)}{2}\,|v(x)-v(y)|^{p-2}\,(v(x)-v(y))\,\big(g(v(x))-g(v(y))\big)\, \left(\psi(x)^p+\psi(y)^p\right)\\
&- \frac{\varepsilon^{-(p-1)}}{2}\,\Big((v(x)+\delta)^{\beta+p-1}+(v(y)+\delta)^{\beta+p-1}\Big)\, |\psi(x)-\psi(y)|^p.
\end{split}
\]
If we now choose $\varepsilon=(4\,(p-1))^{-1}$, use \eqref{tassello2} and recall the definition of $G$, we finally get
\[
\begin{split}
|v(x)-v(y)|^{p-2}&\,(v(x)-v(y))\,(g(v(x))\,\psi(x)^p-g(v(y))\,\psi(y)^p)\\
&\ge c\,|G(v(x))-G(v(y))|^{p}\, \left(\psi(x)^p+\psi(y)^p\right)\\
&- \frac{C}{\beta}\,\left(\frac{\beta+p-1}{p}\right)^p\,\left(G(v(x))^p+G(v(y))^p\right)\, |\psi(x)-\psi(y)|^p,
\end{split}
\]
for some $c=c(p)>0$ and $C=C(p)>0$. We observe that 
\[
\begin{split}
\Big|G(v(x))\,\psi(x)-G(v(y))\,\psi(y)\Big|^p&\le 2^{p-1}\,|G(v(x))-G(v(y))|^p\, \Big(\psi(x)^p+\psi(y)^p\Big)\\
&+2^{p-1}\, \Big(G(v(x))^p+G(v(y))^p\Big)\, |\psi(x)-\psi(y)|^p,
\end{split}
\]
then we finally get
\begin{equation}
\label{I1}
\begin{split}
c\,\int_{\Omega'}\int_{\Omega'} &\frac{\Big|G(v(x))\,\psi(x)-G(v(y))\,\psi(y)\Big|^{p}\,}{|x-y|^{N+s\,p}}\,dx\, dy\\
&\le \mathcal{I}_1+\frac{C}{\beta}\,\left(\frac{\beta+p-1}{p}\right)^p\,\int_{\Omega'}\int_{\Omega'} \frac{|\psi(x)-\psi(y)|^p}{|x-y|^{N+s\,p}}\,\Big(G(v(x))^p+G(v(y))^p\Big)\, dx\,dy.\\
\end{split}
\end{equation}
\vskip.2cm\noindent
{\bf Estimate of $\mathcal{I}_2$.} This is easier to estimate, we simply observe that by monotonicity of $\tau\mapsto |\tau|^{p-2}\,\tau$ and hypothesis \eqref{positività}, for $x\in\Omega'$ we have
\[
|v(x)-v(y)|^{p-2}\, (v(x)-v(y))\ge -|v(y)|^{p-2}\, v(y),
\]
thus we get
\begin{equation}
\label{I2}
\begin{split}
\mathcal{I}_2&\ge -\int_{\mathbb{R}^N\setminus\Omega'}\int_{\Omega'} \frac{|v(y)|^{p-2}\, v(y)}{|x-y|^{N+s\,p}}\, g(v(x))\,\psi(x)^p\, dx\, dy\\
&\ge-\left(\sup_{x\in\mathrm{spt}(\psi)} \int_{\mathbb{R}^N\setminus \Omega'} \frac{|v(y)|^{p-1}}{|x-y|^{N+s\,p}}\,dy\right)\, \int_{\Omega'} g(v(x))\, \psi(x)^p\, dx.
\end{split}
\end{equation}
\vskip.2cm\noindent
{\bf Estimate of $\mathcal{I}_3$.} This is estimated exactly as before, i.e. it is sufficient to use for $y\in\Omega'$
\[
|v(x)-v(y)|^{p-2}\, (v(x)-v(y))\le |v(x)|^{p-2}\, v(x),
\]
then
\begin{equation}
\label{I3}
\begin{split}
\mathcal{I}_3
&\ge -\left(\sup_{y\in\mathrm{spt}(\psi)} \int_{\mathbb{R}^N\setminus \Omega'} \frac{|v(x)|^{p-1}}{|x-y|^{N+s\,p}}\,dx\right)\, \int_{\Omega'} g(v(y))\, \psi(y)^p\, dy.
\end{split}
\end{equation}
{\bf Conclusion.} Since from \eqref{subeigeneq20} we have
\[
\sum_{i=1}^3 \mathcal{I}_1\le L^{p-1}\,\int_\Omega |F| \, g(v)\, \psi^p\, dx,
\]
it is sufficient to use \eqref{I1}, \eqref{I2} and \eqref{I3} in order to get the desired conclusion.
\end{proof}
As in \cite{DKP}, we define the {\it nonlocal tail} of a function $\varphi$ by
\[
\mathrm{Tail}(\varphi;x_0,R):=\left(R^{s\,p} \int_{\mathbb{R}^N\setminus B_R(x_0)} \frac{|\varphi(x)|^{p-1}}{|x-x_0|^{N+s\,p}}\, dx\right)^\frac{1}{p-1}.
\]
Observe that this is a scaling invariant quantity.
An important consequence of the Caccioppoli inequality is the following result for solutions (see \cite[Lemma 2.2]{KMS} for the case $F\equiv 0$). We state the result for the subconformal case $s\,p<N$, then it will be evident how to adapt the argument to cover the case $s\,p=N$, where $F\in L^q(\Omega)$ with $q>1$.
\begin{coro}
\label{coro:gajardo!}
Let $1<p<\infty$ and $0<s<1$ such that $s\,p<N$ and let $F\in L^{(p^*)'}(\Omega)$. We take $0<r<R$ such that
$B_{r}(x_0)\Subset B_{R}(x_0)\Subset\Omega$. For the solution $u$ of \eqref{eigeneq} there holds
\begin{equation}
\label{caccioppolivera}
\begin{split}
\int_{B_r} &\int_{B_r} \frac{\left|u(x)-u(y)\right|^{p}\,}{|x-y|^{N+s\,p}}\, dx\, dy\\
&\le\, \frac{C}{R^{s\,p}}\,\left(\frac{R}{R-r}\right)^{N+s\,p+p}\,  \left\{\|u\|^p_{L^p(B_R)}+R^N\,\Big[\mathrm{Tail}(u;x_0,R)\Big]^p\right\}+C\, \|F\|^\frac{p}{p-1}_{L^{(p^*)'}(B_R)},
\end{split}
\end{equation}
for some constant $C=C(N,p)>0$.
\end{coro}
\begin{proof}
In what follows, we omit to indicate the center $x_0$. In \eqref{monster} we choose $\Omega'=B_{R}$ and $\psi\in C^\infty_0(B_{(r+R)/2})$ such that
\[
0\le \psi\le 1\quad \mbox{ in } B_{(r+R)/2},\qquad \psi\equiv 1 \mbox{ in } B_{r}\qquad \mbox{ and }\qquad |\nabla \psi|\le \frac{C}{R-r},
\]
for some universal constant $C>0$. We also take $v=u_+$ the positive part of $u$ and $g(t)=t$, i.e. $\beta=1$ and $\delta=0$. Then from \eqref{monster}, using the Lipschitz character of $\psi$ we get
 \begin{equation}
\label{caccioKMS}
\begin{split}
\int_{B_{R}}&\int_{B_{R}} \frac{\left|u_+(x)\,\psi(x)-u_+(y)\,\psi(y)\right|^{p}\,}{|x-y|^{N+s\,p}}\, dx\, dy+\frac{1}{R^{s\,p}}\,\int_{B_R} u_+^p\,\psi^p\,dx\\
&\le \frac{C}{(R-r)^p}\,\int_{B_{R}}\int_{B_{R}} \frac{1}{|x-y|^{N+s\,p-p}}\,\left(u_+(x)^p+u_+(y)^p\right)\, dx\,dy+\frac{1}{R^{s\,p}}\,\int_{B_R} u_+^p\,\psi^p\,dx\\
&+C\,\left(\sup_{y\in B_{\frac{R+r}{2}}} \int_{\mathbb{R}^N\setminus B_{R}} \frac{u_+^{p-1}}{|x-y|^{N+s\,p}}\,dx\right)\, \int_{B_{R}} u_+\, \psi^p\, dx+\int_\Omega |F|\, u_+\,\psi^p\, dx.
\end{split}
\end{equation}
where we added the term $R^{-s\,p}\,\int_{B_{R}} u_+^p\,\psi^p\, dx$ on both sides. 
We then observe that
\begin{equation}
\label{pre1}
\begin{split}
\int_{B_{R}}\int_{B_{R}} &\frac{1}{|x-y|^{N+s\,p-p}}\,\left(u_+(x)^p+u_+(y)^p\right)\, dx\,dy\\
&=2\, \int_{B_{R}}\int_{B_{R}} \frac{u_+(x)^p}{|x-y|^{N+s\,p-p}}\, dx\,dy\\
&\le 2\,\int_{B_{R}}\left(\int_{B_{2\,R}} \frac{dy}{|x_0-y|^{N+s\,p-p}}\right)\, u_+^p\, dx=\frac{c}{1-s}\, R^{p\,(1-s)}\, \|u_+\|^p_{L^p(B_R)},\\
\end{split}
\end{equation}
for some constant $c=c(N,p)>0$. For the second term on the right-hand side of \eqref{caccioKMS}, we have 
\begin{equation}
\label{pre2}
\begin{split}
\sup_{y\in B_{(R+r)/2}} \int_{\mathbb{R}^N\setminus B_{R}} \frac{u_+^{p-1}}{|x-y|^{N+s\,p}}\,dx&\le \frac{2^{N+s\,p}\,R^{N+s\,p}}{(R-r)^{N+s\,p}} \int_{\mathbb{R}^N\setminus B_{R}} \frac{u_+^{p-1}}{|x-x_0|^{N+s\,p}}\,dx\\
&=\frac{2^{N+s\,p}\,R^{N}}{(R-r)^{N+s\,p}}\,\Big[\mathrm{Tail}(u_+;x_0,R)\Big]^{p-1},
\end{split}
\end{equation}
since for every $y\in B_{(R+r)/2}$ and $x\in\mathbb{R}^N\setminus B_{R}$ we have
\[
\begin{split}
|x-y|\ge |x-x_0|-|x_0-y|&\ge |x-x_0|-\frac{R+r}{2}
\ge\frac{R-r}{2\,R}\, |x-x_0|.
\end{split}
\]
Finally, by H\"{o}lder and Young inequalities
\[
\begin{split}
\int_{\Omega} |F|\,u_+\,\psi^p\, dx&\le \frac{(p-1)}{p\,\,\tau^\frac{1}{p-1}}\, \left(\int_{B_R} \left(|F|\,\psi^{p-1}\right)^{(p^*)'}\, dx\right)^\frac{p'}{(p^*)'}+\frac{\tau}{p}\, \left(\int_{B_R} (u_+\,\psi)^{p^*}\, dx\right)^\frac{p}{p^*}.
\end{split}
\]
For the last term we use Sobolev inequality \eqref{sobolev}
and choose $\tau\ll 1$ small enough, in order to absorb it in the left-hand side of \eqref{caccioKMS}. Then from \eqref{caccioKMS}, \eqref{pre1} and \eqref{pre2} we get
\[
\begin{split}
\int_{B_{R}}&\int_{B_{R}} \frac{\left|u_+(x)\,\psi(x)-u_+(y)\,\psi(y)\right|^{p}\,}{|x-y|^{N+s\,p}}\, dx\, dy+\frac{1}{R^{s\,p}}\,\int_{B_R} u_+^p\, \psi^p\,dx\\
&\le \frac{C}{R^{s\,p}}\,\left(\frac{R}{R-r}\right)^p\,\|u_+\|^p_{L^p(B_R)}\\
&+\frac{C}{R^{s\,p}}\,\left(\frac{R}{R-r}\right)^{N+s\,p}\,\Big[\mathrm{Tail}(u_+;x_0,R)\Big]^{p-1}\, \|u_+\|_{L^1(B_R)}+C\, \|F\|^\frac{p}{p-1}_{L^{(p^*)'}(B_R)}.
\end{split}
\]
By H\"{o}lder and Young inequalities we have
\[
\Big[\mathrm{Tail}(u_+;x_0,R)\Big]^{p-1}\, \|u_+\|_{L^1(B_R)}\le \frac{(p-1)\,\omega_N\,R^N}{p}\,\Big[\mathrm{Tail}(u_+;x_0,R)\Big]^{p}+\frac{1}{p}\,\|u_+\|^p_{L^p(B_R)}.
\]
By using that $\psi\equiv 1$ on $B_r$ and also $R/(R-r)>1$ we thus get
\[
\begin{split}
\int_{B_r} \int_{B_r} &\frac{\left|u_+(x)-u_+(y)\right|^{p}\,}{|x-y|^{N+s\,p}}\, dx\, dy\\
&\le \frac{C}{R^{s\,p}} \left(\frac{R}{R-r}\right)^{N+s\,p+p} \left\{\|u_+\|^p_{L^p(B_R)}+R^N\,\Big[\mathrm{Tail}(u_+;x_0,R)\Big]^p\right\}+C\, \|F\|^\frac{p}{p-1}_{L^{(p^*)'}(B_R)},
\end{split}
\]
for some constant $C=C(N,p,s)>0$. This proves \eqref{caccioppolivera}. By reproducing the proof above for $v=u_-$, the negative part of $u$, and then summing up the resulting inequalities, we finally get inequality \eqref{caccioppolivera} for the solution itself.
\end{proof}
\begin{oss}
\label{oss:medie}
We observe that the previous estimate still holds for $u-c$, where $c\in\mathbb{R}$. In particular, if we take the average $c=\overline u_{x_0,R}$, we get
\begin{equation}
\label{caccioKMSbis}
\begin{split}
\int_{B_r}& \int_{B_r} \frac{\left|u(x)-u(y)\right|^{p}\,}{|x-y|^{N+s\,p}}\, dx\, dy\le\, \frac{C}{R^{s\,p}}\,\left(\frac{R}{R-r}\right)^{N+s\,p+p}\, \\
&\times \left\{\|u-\overline u_{x_0,R}\|^p_{L^p(B_R)}+R^N\,\Big[\mathrm{Tail}(u-\overline u_{x_0,R};x_0,R)\Big]^p\right\}+C\, \|F\|^\frac{p}{p-1}_{L^{(p^*)'}(B_R)}.
\end{split}
\end{equation}
\end{oss}

\subsection{Local $L^\infty$ estimates}

In the homogeneous case, the local $L^\infty$ estimate below has been proved in \cite[Theorem 1.1]{DKP}. The proof in \cite{DKP} uses a De Giorgi-type iteration. The case $F\not\equiv 0$ has been treated in \cite[Theorem 1.2]{KMS}, by using a perturbative argument.
The result in \cite{KMS} provides a pointwise nonlinear potential estimate on the solution.
\par
Our approach is more elementary and direct, it relies on the Caccioppoli inequality of Proposition \ref{prop:caccioloc} and a Moser's iteration.
\begin{teo}[Local boundedness]
\label{teo:loc_bound}
Let $1<p<\infty$ and $0<s<1$ such that $s\,p\le N$. Let $F\in L^q(\Omega)$ for $q>N/(s\,p)$ and $v=f(u)$ be a subsolution, i.e. a function satisfying \eqref{subeigeneq2} with $f$ convex and $L-$Lipschitz. We take $0<r_0<R_0$ such that
$B_{r_0}(x_0)\subset B_{R_0}(x_0)\Subset\Omega$,
and suppose that 
\[
v\ge 0\qquad \mbox{ on } B_{R_0}(x_0).
\]
Then the following scaling invariant estimate holds
\begin{equation}
\label{tremilamadonne}
\begin{split}
\|v\|_{L^\infty(B_{r_0})}&\le C\left[\left(\fint_{B_{R_0}} v^p\, dx\right)^\frac{1}{p}+\left(\frac{R_0}{r_0}\right)^\frac{s\,p}{p-1}\mathrm{Tail}(v;x_0,r_0)+L\left(R_0^{s\,p-\frac{N}{q}}\,\|F\|_{L^q(B_{R_0})}\right)^\frac{1}{p-1}\right],
\end{split}
\end{equation}
where 
\[
C=C_0\,\left(\frac{R_0}{r_0}\right)^\frac{N}{p}\,\left(\frac{R_0}{R_0-r_0}\right)^{(N+s\,p+p)\,\frac{N}{s\,p^2}}\qquad \mbox{ and }\qquad C_0=C_0(N,s,p,q)>0.
\]
\end{teo}
\begin{proof}
We use Proposition \ref{prop:caccioloc} in order to perform a suitable Moser's interation.  We first prove \eqref{tremilamadonne} for $r_0=1$ and $R_0=\gamma>1$, then we will get \eqref{tremilamadonne} for general $r_0<R_0$ by a scaling argument. It is also clear that it is sufficient to prove the result for $L=1$, up to replace $v$ by $v/L$ and use homogeneity of the operator.
\vskip.2cm\noindent
{\it Estimate at scale $1$.}
We thus take $1=r_0\le r<R\le \gamma$, choose $\Omega'=B_{R}$ and $\psi\in C^\infty_0(B_{(r+R)/2})$ such that
\[
0\le \psi\le 1\quad \mbox{ in } B_{(r+R)/2},\qquad \psi\equiv 1 \mbox{ in } B_{r}\qquad \mbox{ and }\qquad |\nabla \psi|\le \frac{C}{R-r},
\]
for some universal constant $C>0$. For notational simplicity, for $\delta>0$ we also set
\[
v_\delta=v+\delta.
\]
From \eqref{monster} we get
\begin{equation}
\label{ready?}
\begin{split}
\beta\,&\left(\frac{p}{\beta+p-1}\right)^p\,\int_{B_{R}}\int_{B_{R}} \frac{\left|v_\delta(x)^\frac{\beta+p-1}{p}\,\psi(x)-v_\delta(y)^\frac{\beta+p-1}{p}\,\psi(y)\right|^{p}\,}{|x-y|^{N+s\,p}}\, dx\, dy\\
&\le \frac{C}{(R-r)^p}\,\int_{B_{R}}\int_{B_{R}} \frac{1}{|x-y|^{N+s\,p-p}}\,\left(v_\delta(x)^{\beta+p-1}+v_\delta(y)^{\beta+p-1}\right)\, dx\,dy\\
&+C\,\left(\sup_{y\in B_{\frac{R+r}{2}}} \int_{\mathbb{R}^N\setminus B_{R}} \frac{|v(x)|^{p-1}}{|x-y|^{N+s\,p}}\,dx\right)\, \int_{B_{R}} v_\delta^\beta\, \psi^p\, dx+\int_\Omega |F|\, v_\delta^\beta\,\psi^p\, dx.
\end{split}
\end{equation}
We then observe that by proceeding as in \eqref{pre1}, we have 
\begin{equation}
\label{1}
\begin{split}
\int_{B_{R}}\int_{B_{R}} &\frac{\left(v_\delta(x)^{\beta+p-1}+v_\delta(y)^{\beta+p-1}\right)}{|x-y|^{N+s\,p-p}}\, dx\,dy\le \frac{c}{1-s}\, R^{p\,(1-s)}\, \int_{B_{R}} v_\delta^{\beta+p-1}\, dx,\\
\end{split}
\end{equation}
for some constant $c=c(N,p)>0$.
As for the second term in the right-hand side of \eqref{ready?}, as in \eqref{pre2} we have (recall that $r_0=1$)
\begin{equation}
\label{2}
\begin{split}
\sup_{y\in B_{(R+r)/2}} \int_{\mathbb{R}^N\setminus B_{R}} \frac{|v(x)|^{p-1}}{|x-y|^{N+s\,p}}\,dx
&\le \frac{2^{N+s\,p}\,R^{N+s\,p}}{(R-r)^{N+s\,p}}\Big[\mathrm{Tail}(v;x_0,1)\Big]^{p-1}.
\end{split}
\end{equation}
By collecting \eqref{1} and \eqref{2} in \eqref{ready?}, we get
\[
\begin{split}
\beta\,\left(\frac{p}{\beta+p-1}\right)^p&\,\int_{B_{R}}\int_{B_{R}} \frac{\left|v_\delta(x)^\frac{\beta+p-1}{p}\,\psi(x)-v_\delta(y)^\frac{\beta+p-1}{p}\,\psi(y)\right|^{p}\,}{|x-y|^{N+s\,p}}\, dx\, dy\\
&\le \frac{C}{R^{s\,p}}\,\left(\frac{R}{R-r}\right)^p\, \int_{B_{R}} v_\delta^{\beta+p-1}\, dx\\
&+C\,\left(\frac{R}{R-r}\right)^{N+s\,p}\,\Big[\mathrm{Tail}(v;x_0,1)\Big]^{p-1}\,\int_{B_{R}} v_\delta^\beta\, \psi^p\, dx+\int_{B_R} |F|\, v_\delta^\beta\,\psi^p\, dx.
\end{split}
\] 
If we add the term 
\[
\beta\,\left(\frac{p}{\beta+p-1}\right)^p\,\frac{1}{R^{s\,p}}\,\int_{B_{R}} v_\delta^{\beta+p-1}\,\psi^p\, dx,
\]
on both sides and use again the Sobolev inequality \eqref{sobolev}, we gain\footnote{We use again that 
\(
\beta\,\left(\frac{p}{\beta+p-1}\right)^p \le 1
\)
on the right-hand side, see Lemma \ref{betaandpi}.} 
\begin{equation}
\label{piano}
\begin{split}
\beta\,\left(\frac{p}{\beta+p-1}\right)^p\, &\left(\int_{B_R} \left(v_\delta^\frac{\beta+p-1}{p}\,\psi\right)^{p^*}\, dx\right)^\frac{p}{p^*}\\
&\le \frac{C}{R^{s\,p}}\,\left[\left(\frac{R}{R-r}\right)^p+1\right]\,\int_{B_R} v_\delta^{\beta+p-1}\, dx+C\,\int_\Omega |F|\, v_\delta^\beta\,\psi^p\, dx\\
&+C\,\left(\frac{R}{R-r}\right)^{N+s\,p}\,\Big[\mathrm{Tail}(v;x_0,1)\Big]^{p-1}\, \int_{B_{R}} v_\delta^\beta\, dx,\\
\end{split}
\end{equation}
for some constant $C=C(N,s,p)>0$. For the term containing $F$ in \eqref{piano}, we use
\begin{equation}
\label{delta}
\delta^{p-1}\,v_\delta^\beta\le  v_\delta^{\beta+p-1},
\end{equation}
H\"{o}lder inequality and Lebesgue interpolation\footnote{Observe that we have $p<p\,q'<p^*$, thanks to the assumption on $q$.}.
These yield, for a parameter $0<\tau\ll 1$
\[
\begin{split}
\int_\Omega |F|\, v_\delta^\beta\,\psi^p\, dx&\le \delta^{1-p}\, \int_\Omega |F|\, v_\delta^{\beta+p-1}\,\psi^p\, dx\le \delta^{1-p}\,\|F\|_{L^q(B_R)}\,\left(\int_{B_R} \left(v_\delta^\frac{\beta+p-1}{p}\,\psi\right)^{p\,q'}\, dx\right)^\frac{1}{q'}\\
&\le c\,\tau\,\delta^{1-p}\,\|F\|_{L^q(B_R)}\,\left(\int_{B_R} \left(v_\delta^\frac{\beta+p-1}{p}\,\psi\right)^{p^*}\, dx\right)^\frac{p}{p^*}\\
&+c\,\tau^{-\frac{N}{s\,p\,q-N}}\,\delta^{1-p}\,\|F\|_{L^q(B_R)}\,\int_{B_R} v_\delta^{\beta+p-1}\,\psi^p\, dx,
\end{split}
\]
for $c=c(N,s,p,q)>0$.
By choosing $\tau$ as follows
\[
\tau=\frac{\delta^{p-1}}{2\,c}\,\frac{1}{\|F\|_{L^q(B_R)}}\, \beta\,\left(\frac{p}{\beta+p-1}\right)^p,
\]
we can absorb the first term. Observing that for $\beta \ge 1$, we have
\[
\beta\,\left(\frac{p}{\beta+p-1}\right)^p\ge \left(\frac{p}{\beta+p-1}\right)^{p-1},
\]
from \eqref{piano} and \eqref{delta} we can infer
\[
\begin{split}
\left\|v_\delta\right\|^\vartheta_{L^{\vartheta\,p^*/p}(B_r)}&\le
\frac{C}{R^{s\,p}}\,\left(\frac{\vartheta}{p}\right)^{p-1}\,  \left\|v_\delta\right\|^\vartheta_{L^{\vartheta}(B_R)}\\
&\times\,\left\{\left(\frac{R}{R-r}\right)^p+\left(\frac{R}{R-r}\right)^{N+s\,p}\,\left[\gamma^\frac{s\,p}{p-1}\frac{\mathrm{Tail}(v;x_0,1)}{\delta}\right]^{p-1}\right.\\
&+\left.\gamma^{s\,p}\,\left(\frac{\|F\|_{L^q(B_R)}}{\delta^{p-1}}\right)^{\frac{s\,p\,q}{s\,p\,q-N}}\,\left(\frac{\vartheta}{p}\right)^{(p-1)\,\frac{N}{s\,p\,q-N}}\right\},
\end{split}
\]
where we set $\vartheta=\beta+p-1$. By observing that $R/(R-r)>1$ and $\vartheta/p\ge 1$, from the previous we can also obtain
\begin{equation}
\label{ehi}
\begin{split}
\left\|v_\delta\right\|^\vartheta_{L^{\vartheta\,p^*/p}(B_r)}&\le
\frac{C}{R^{s\,p}}\,\left(\frac{R}{R-r}\right)^{N+s\,p+p}\,\mathcal{T}(\delta)\,\left(\frac{\vartheta}{p}\right)^{\eta}\,  \left\|v_\delta\right\|^\vartheta_{L^{\vartheta}(B_R)},
\end{split}
\end{equation}
where $C=C(N,p,s,q)>0$. In the previous estimate we set for simplicity
\[
\eta=(p-1)\,\frac{s\,p\,q}{s\,p\,q-N}\qquad \mbox{ and }\qquad
\mathcal{T}(\delta)=1+\left[\gamma^\frac{s\,p}{p-1}\frac{\mathrm{Tail}(v;x_0,1)}{\delta}\right]^{p-1} +\left(\gamma^\frac{s\,p}{\eta}\frac{\|F\|^\frac{1}{p-1}_{L^q(B_\gamma)}}{\delta}\right)^\eta.
\]
If we define $\chi=p^*/p>1$ and use that $R>r_0=1$, then from \eqref{ehi} we have
\begin{equation}
\label{ready!}
\begin{split}
\big\|v_\delta\big\|_{L^{\vartheta\,\chi}(B_{r})}&\le \left(C\,\left(\frac{R}{R-r}\right)^{N+s\,p+p}\,\mathcal{T}(\delta)\right)^\frac{1}{\vartheta}\,\left(\vartheta^\frac{1}{\vartheta}\right)^\eta\,\big\|v_\delta\big\|_{L^\vartheta(B_{R})}.
\end{split}
\end{equation}
We now define the sequences $\{\vartheta_k\}_{k\in\mathbb{N}}$ and  $\{\varrho_k\}_{k\in\mathbb{N}}$ by
\[
\vartheta_k=\chi\,\vartheta_{k-1}=\left(\frac{N}{N-s\,p}\right)^k\, p,\qquad\qquad \varrho_k=1+\frac{\gamma-1}{2^k}.
\]
Observe that 
\[
\left(\frac{\varrho_k}{\varrho_k-\varrho_{k+1}}\right)^{N+s\,p+p}\le \left(2^{k+1}\right)^{N+s\,p+p}\,\left(\frac{\gamma}{\gamma-1}\right)^{N+s\,p+p}, 
\]
thus \eqref{ready!} applied with $r=\varrho_{k+1}$ and $R=\varrho_k$ gives
\begin{equation}
\label{ready!!}
\begin{split}
\big\|v_\delta\big\|_{L^{\vartheta_{k+1}}(B_{\varrho_{k+1}})}&\le\left(C'\left(\frac{\gamma}{\gamma-1}\right)^{N+s\,p+p}\mathcal{T}(\delta)\right)^\frac{1}{\vartheta_k}\,\left(\vartheta_k^\eta\,2^{k\,(N+(1+s)\,p)}\right)^\frac{1}{\vartheta_k}\,\big\|v_\delta\big\|_{L^{\vartheta_k}(B_{\varrho_k})},
\end{split}
\end{equation}
where $C=C(N,s,p,q)>0$. We now observe that 
\[
\sum_{k=0}^\infty \frac{1}{\vartheta_k}
=\frac{N}{s\,p^2} \qquad\mbox{ and }\qquad \lim_{n\to\infty}\left(\prod_{k=0}^n \vartheta_k^\frac{\eta}{\vartheta_k}\,\left(2^{N+(1+s)\,p}\right)^\frac{k}{\vartheta_k}\right)<+\infty.
\]
Then by starting from $k=0$ and iterating \eqref{ready!!} infinitely many times, we finally get
\[
\begin{split}
\big\|v_\delta\big\|_{L^\infty(B_{1})}&\le C\,\left(\left(\frac{\gamma}{\gamma-1}\right)^{N+s\,p+p}\,\mathcal{T}(\delta)\right)^\frac{N}{s\,p^2}\,\big\|v_\delta\big\|_{L^{p}(B_{\gamma})},
\end{split}
\]
for some constant $C=C(N,s,p,q)>0$. If we now recall the definition of $\mathcal{T}(\delta)$, choose
\[
\delta=\gamma^\frac{s\,p}{p-1}\,\mathrm{Tail}(v;x_0,1)+\|F\|^\frac{1}{p-1}_{L^q(B_{\gamma})}\, \gamma^\frac{s\,p}{\eta},
\]
and use the triangle inequality, from the previous we finally get \eqref{tremilamadonne} for $r_0=1$ and $R_0=\gamma>1$.
\vskip.2cm\noindent
{\it Estimate at general scales.} We now take $r_0<R_0$ and define
\[
\gamma=\frac{R_0}{r_0}\qquad \mbox{ and }\qquad \lambda=r_0.
\]
The function $v_\lambda(x):=v(\lambda\,x)$ satisfies \eqref{subeigeneq1} with right-hand side $F_\lambda(x):=\lambda^{s\,p}\, F(\lambda\,x)$. It is now sufficient to use \eqref{tremilamadonne} for $v_\lambda$ with balls $B_1$ and $B_\gamma$, together with the facts
\[
\|v_\lambda\|_{L^\infty(B_1)}=\|v\|_{L^\infty(B_{r_0})},\qquad \int_{B_\gamma} v_\lambda^p\, dx=\left(\frac{R_0}{r_0}\right)^N\,\frac{1}{R_0^N}\,\int_{B_{R_0}} v^p\, dx,
\]
\[
\mathrm{Tail}(v_\lambda;x_0,1)=\mathrm{Tail}(v;x_0,r_0)\qquad \mbox{ and }\qquad
\gamma^{s\,p-\frac{N}{q}}\,\|F_\lambda\|_{L^q(B_{\gamma})}=R_0^{s\,p-\frac{N}{q}}\, \|F\|_{L^q(B_{R_0})}.
\]
These give the desired conclusion.
\end{proof}

\subsection{Basic comparison estimates}

We take $B_r:=B_r(x_0)\Subset\Omega$ and consider the function $v\in W^{s,p}(\mathbb{R}^N)$ such that $v-u\in \widetilde W^{s,p}_0(B_r)$ and which solves the homogeneous equation
\begin{equation}
\label{0}
\int_{\mathbb{R}^N}\int_{\mathbb{R}^N} \frac{|v(x)-v(y)|^{p-2}\, (v(x)-v(y))}{|x-y|^{N+s\,p}}\, (\varphi(x)-\varphi(y))\, dx\, dy=0,
\end{equation}
for every $\varphi\in \widetilde W^{s,p}_0(B_r)$. We point out that $v$ is a locally H\"older function (\cite[Theorem 1.4]{DKP}) and satisfies a suitable Harnack inequality (see \cite{DKP2}).
We define
\[
w=u-v\in \widetilde W^{s,p}_0(B_r),
\]
then we have the following basic estimate. As usual, we consider only the subconformal case $s\,p<N$, the reader can adapt it to the conformal case $s\,p=N$, with $F\in L^q(\Omega)$ for some $q>1$.
\begin{lm}
\label{lm:lemma34}
Let $1<p<\infty$ and $0<s<1$ such that $s\,p<N$. For $F\in L^{(p^*)'}(\Omega)$ and every $B_{2\,r}:=B_{2\,r}(x_0)\Subset\Omega$, with the notation above there holds
\begin{equation}
\label{base}
\begin{split}
\left(\int_{B_{2\,r}} \int_{B_{2\,r}}\right. & \left. \frac{|w(x)-w(y)|^p}{|x-y|^{N+s\,p}}\, dx\,dy\right)^\frac{1}{p}\\
&\le \max\{2-p,0\}\,C\,\|F\|_{L^{(p^*)'}(B_r)}\,\left(\int_{B_{2\,r}}\int_{B_{2\,r}} \frac{|u(x)-u(y)|^p}{|x-y|^{N+s\,p}}\, dx\,dy\right)^\frac{2-p}{p}\\
&+C\,\left(\|F\|_{L^{(p^*)'}(B_r)}\right)^\frac{1}{p-1},
\end{split}
\end{equation}
for some constant $C=C(N,p,s)>0$.
\end{lm}
\begin{proof}
We will use for simplicity the notation $J_p(t)=|t|^{p-2}\, t$.
We test equations \eqref{eigeneq} and \eqref{0} with $\varphi=w$, subtracting them we get
\begin{equation}
\label{sottrai}
\int_{\mathbb{R}^N}\int_{\mathbb{R}^N} \frac{J_p(u(x)-u(y))-J_p(v(x)-v(y))}{|x-y|^{N+s\,p}}\, (w(x)-w(y))\, dx\, dy\\
=\int_\Omega F\,w\,dx.
\end{equation}
We now distinguish two cases.
\vskip.2cm\noindent
{\it Case $1<p<2$.}
Thanks to \eqref{due}
\[
\begin{split}
|w(x)-w(y)|^p&\le C\,\Big((J_p(u(x)-u(y))-J_p(v(x)-v(y)))\,(w(x)-w(y))\Big)^\frac{p}{2}\\
&\times\Big((u(x)-u(y))^2+(v(x)-v(y))^2\Big)^{\frac{2-p}{2}\,\frac{p}{2}},
\end{split}
\]
thus by integrating this over $B_{2\,r}$, using H\"{o}lder inequality and \eqref{sottrai} we get
\[
\begin{split}
\int_{B_{2\,r}} \int_{B_{2\,r}} &\frac{|w(x)-w(y)|^p}{|x-y|^{N+s\,p}}\, dx\,dy\\
&\le c\,\left(\int_{B_{2\,r}}\int_{B_{2\,r}} \frac{J_p(u(x)-u(y))-J_p(v(x)-v(y))}{|x-y|^{N+s\,p}}\, (w(x)-w(y))\, dx\, dy\right)^\frac{p}{2}\\
&\times \left(\int_{B_{2\,r}}\int_{B_{2\,r}} \frac{\left(|u(x)-u(y)|^2+|v(x)-v(y)|^2\right)^\frac{p}{2}}{|x-y|^{N+s\,p}}\, dx\,dy\right)^\frac{2-p}{2}\\
&\le c\,\left(\|F\|_{L^{(p^*)'}(B_r)}\, \|w\|_{L^{p^*}(B_r)}\right)^\frac{p}{2}\, \left(\int_{B_{2\,r}}\int_{B_{2\,r}} \frac{\left(|u(x)-u(y)|^2+|v(x)-v(y)|^2\right)^\frac{p}{2}}{|x-y|^{N+s\,p}}\, dx\,dy\right)^\frac{2-p}{2}.
\end{split}
\]
Observe that in the right-hand side of \eqref{sottrai} we used that $w\equiv 0$ outside $B_r$.
Since $v=u-w$ and by subadditivity of $\tau\mapsto\tau^{p/2}$, we have
\[
\begin{split}
\int_{B_{2\,r}}\int_{B_{2\,r}}& \frac{\left(|u(x)-u(y)|^2+|v(x)-v(y)|^2\right)^\frac{p}{2}}{|x-y|^{N+s\,p}}\, dx\,dy\\
&\le C\,\int_{B_{2\,r}}\int_{B_{2\,r}} \frac{|u(x)-u(y)|^p}{|x-y|^{N+s\,p}}\, dx\,dy+C\,\int_{B_{2\,r}}\int_{B_{2\,r}} \frac{|w(x)-w(y)|^p}{|x-y|^{N+s\,p}}\, dx\,dy,
\end{split}
\]
for some constant $C=C(p)>0$. Up to now, we obtained
\[
\begin{split}
\int_{B_{2\,r}} \int_{B_{2\,r}} &\frac{|w(x)-w(y)|^p}{|x-y|^{N+s\,p}}\, dx\,dy\\
&\le C\,\left(\|w\|_{L^{p^*}(B_r)}\right)^\frac{p}{2}\,\,\left(\|F\|^\frac{p}{2-p}_{L^{(p^*)'}(B_r)}\right)^\frac{2-p}{2}\\
&\times\,\left(\int_{B_{2\,r}}\int_{B_{2\,r}} \frac{|u(x)-u(y)|^p}{|x-y|^{N+s\,p}}\, dx\,dy+\int_{B_{2\,r}}\int_{B_{2\,r}} \frac{|w(x)-w(y)|^p}{|x-y|^{N+s\,p}}\, dx\,dy\right)^\frac{2-p}{2},
\end{split}
\]
possibly for a different $C=C(p)>0$. If we now use Young inequality with exponent $2$ and the continuous embedding $\widetilde W^{s,p}_0(B_{2\,r})\hookrightarrow L^{p^*}(B_{r})$ of Proposition \ref{prop:sobolev_allargati}, 
from the previous and subadditivity of $\tau\mapsto\tau^{2-p}$ we get
\[
\begin{split}
&\int_{B_{2\,r}} \int_{B_{2\,r}} \frac{|w(x)-w(y)|^p}{|x-y|^{N+s\,p}}\, dx\,dy\\
&\le c\,\|F\|^p_{L^{(p^*)'}(B_r)}\left[\left(\int_{B_{2\,r}}\int_{B_{2\,r}} \frac{|u(x)-u(y)|^p}{|x-y|^{N+s\,p}}\, dx\,dy\right)^{2-p}+\left(\int_{B_{2\,r}}\int_{B_{2\,r}} \frac{|w(x)-w(y)|^p}{|x-y|^{N+s\,p}}\, dx\,dy\right)^{2-p}\right].
\end{split}
\]
It is now sufficient to estimate the last term by means of Young inequality with exponents  $1/(p-1)$ and $1/(2-p)$ in order to conclude.
\vskip.2cm\noindent
{\it Case $p\ge 2$.} This case is simpler. In this case \eqref{uno} implies
\[
\begin{split}
|w(x)-w(y)|^p&\le C\,\Big(J_p(u(x)-u(y))-J_p(v(x)-v(y))\Big)\,\Big(w(x)-w(y))\Big),
\end{split}
\]
and thus from \eqref{sottrai} we get
\[
\begin{split}
\int_{B_{2\,r}} \int_{B_{2\,r}} &\frac{|w(x)-w(y)|^p}{|x-y|^{N+s\,p}}\, dx\,dy \le \int_{B_r} F\, w\, dx\\
&\le \|F\|_{L^{(p^*)'}(B_r)}\,\|w\|_{L^{p^*}(B_r)}\le \frac{(p-1)\,\tau^\frac{1}{1-p}}{p}\,\|F\|^\frac{p}{p-1}_{L^{(p^*)'}(B_r)}+\frac{\tau}{p}\, \|w\|^p_{L^{p^*}(B_r)}.
\end{split}
\]
It is now sufficient to use again the Sobolev inequality of Proposition \ref{prop:sobolev_allargati} to conclude.
\end{proof}  
\begin{oss}
The previous Lemma is the analogous of \cite[Lemmas 3.2 \& 3.4]{KMS}, under the simplified assumption that the right-hand side $F$ belongs to $L^{(p^*)'}(\Omega)$. For $1<p<2$, the further restriction $p>2-s/N$ is needed in \cite{KMS}, since the right-hand side $F$ is a measure not necessarily belonging to the relevant dual Sobolev space. On the contrary, under the standing assumption of this section we do not need this restriction on $p$.
\end{oss}
We now introduce as in \cite{KMS} the {\it global excess} functional
\[
\mathcal{E}(u;x_0,r)=\left(\fint_{B_r(x_0)} |u-\overline u_{x_0,r}|^p\, dx\right)^\frac{1}{p}+\mathrm{Tail}(u-\overline u_{x_0,r};x_0,r),
\]  
where $\overline u_{x_0,r}$ is the average of $u$ over the ball $B_r(x_0)$, recall \eqref{media}.
If we appeal to Corollary \ref{coro:gajardo!} and Remark \ref{oss:medie}, we get the following.
\begin{lm}
\label{lm:rimpiazzo}
Let $1<p<\infty$ and $0<s<1$ be such that $s\,p<N$ and let $F\in L^{(p^*)'}(\Omega)$.
Then for every $B_{2\,r}(x_0)\Subset B_R(x_0)\Subset\Omega$, with the notation above there holds
\begin{equation}
\label{basebis}
\begin{split}
\left(\int_{B_{2\,r}} \right.&\left.\int_{B_{2\,r}} \frac{|w(x)-w(y)|^p}{|x-y|^{N+s\,p}}\, dx\,dy\right)^\frac{1}{p}\\
&\le \max\{2-p,0\}\,C\,\|F\|_{L^{(p^*)'}(B_{R})}\,\left[{R^{\frac{N}{p}-s}}\,\left(\frac{R}{R-2\,r}\right)^{\frac{N}{p}+s+1}\right]^{2-p}\,\mathcal{E}(u;x_0,R)^{2-p}\\
&+C\,\left(\|F\|_{L^{(p^*)'}(B_{R})}\right)^\frac{1}{p-1},\\
\end{split}
\end{equation}
for some constant $C=C(N,p,s)>0$. Moreover, we also have
\begin{equation}
\label{basebis_coro}
\begin{split}
\|w\|_{L^{p^*}(B_r)}&\le \max\{2-p,0\}\,\widetilde C\,\|F\|_{L^{(p^*)'}(B_{R})}\,\left[{R^{\frac{N}{p}-s}}\,\left(\frac{R}{R-2\,r}\right)^{\frac{N}{p}+s+1}\right]^{2-p}\,\mathcal{E}(u;x_0,R)^{2-p}\\
&+\widetilde C\,\left(\|F\|_{L^{(p^*)'}(B_{R})}\right)^\frac{1}{p-1},
\end{split}
\end{equation}
where again $\widetilde C=\widetilde C(N,p,s)>0$.
\end{lm}
\begin{proof}
We notice at first that \eqref{basebis_coro} follows from \eqref{basebis} and the Sobolev inequality of Proposition \ref{prop:sobolev_allargati}.
\par
Let us prove \eqref{basebis}.
For $p\ge 2$ there is nothing to prove. For $1<p<2$,
it is sufficient to combine \eqref{caccioKMSbis} and \eqref{base}, then with some simple computations one gets \eqref{basebis}.
\end{proof}
\begin{oss}
The previous Lemma is the analogue of \cite[Lemmas 3.3 \& 3.5]{KMS}, when $F$ belongs to $L^{(p^*)'}(\Omega)$. 
\end{oss}

\subsection{A theorem by Kuusi, Mingione and Sire}


\begin{teo}[\cite{KMS}, Theorem 1.5]
\label{teo:continue!}
Let $1<p<\infty$ and $0<s<1$. If $F\in L^{q}(\Omega)$ with $q>N/(s\,p)$, then the solution $u\in\widetilde W^{s,p}_0(\Omega)$ of \eqref{eigeneq} is continuous.
\end{teo}
\begin{proof}
As already explained, for $p\ge 2$ the continuity result of \cite[Theorem 1.5]{KMS} directly applies to our situation. For the case $1<p<2$, it is sufficient to reproduce the proof of \cite[Theorem 1.5]{KMS}, by replacing their Lemmas 3.3 and 3.5 with our Lemma \ref{lm:rimpiazzo} above.
\par
For the reader's convenience, we briefly explain the main points of the proof by Kuusi, Mingione and Sire, by referring to their paper for the details.
The result in \cite{KMS} is based on the following pointwise estimate\footnote{In \cite{KMS}, the local term in the right-hand side is replaced by the {\it Wolff potential} $\mathbf{W}^F_{s,p}(x_0,r)$, which for a general Borel measure $\mu$ is defined by
\[
\mathbf{W}^\mu_{s,p}(x_0,r):= \int_0^r \left(\frac{|\mu|(B_\varrho(x_0))}{\varrho^{N-s\,p}}\right)^\frac{1}{p-1}\, \frac{d\,\varrho}{\varrho}.
\]
It is standard to see that when $\mu$ has a density $F\in L^{q}$ (here $q>N/(s\,p)$) with respect to the $N-$dimensional Lebesgue measure, we have
\[
\mathbf{W}^\mu_{s,p}(x_0,r)\le c\, \left(\|F\|_{L^q(B_r(x_0))}\, r^{s\,p-\frac{N}{q}}\right)^\frac{1}{p-1},
\]
with $c=c(N,s,p,q)>0$.}
\begin{equation}
\label{wulff}
\left|\overline u_{x_0,r}-u(x_0)\right|\le c\, \left(r^{s\,p-\frac{N}{q}}\,\|F\|_{L^q(B_{r})}\right)^\frac{1}{p-1}+c\,\mathcal{E}(u;x_0,r),
\end{equation}
valid for every ball $B_r:=B_r(x_0)\Subset\Omega$, see \cite[Theorem 1.4]{KMS}. Moreover, $u$ enjoys the following VMO--type property, locally in $\Omega'\Subset\Omega$: for every $\Omega''\Subset\Omega'$
\begin{equation}
\label{VMO}
\lim_{\varrho_0\to 0} \left[\sup_{x_0\in \Omega''} \sup_{0<\varrho<\varrho_0} \mathcal{E}(u;x_0,\varrho)\right]=0,
\end{equation}
see \cite[Section 6]{KMS}.
Then \eqref{wulff} and \eqref{VMO} imply that $u$ is the uniform limit of the net of continuous functions $\{\overline u_{x_0,r}\}_{r}$, thus $u$ itself is continuous. In turn, both the proof \eqref{wulff} and that of \eqref{VMO} are based on the crucial decay estimate (here $\varrho\le r$)
\begin{equation}
\label{continua?}
\begin{split}
\mathcal{E}(u;x_0,\sigma\,\varrho)&\le c\,\sigma^\alpha\,\left(\frac{\varrho}{r}\right)^s\, \mathcal{E}(u;x_0,r)+c\,\sigma^\alpha\,\int_\varrho^r \left(\frac{\varrho}{t}\right)^s\,\mathcal{E}(u;x_0,t)\,\frac{dt}{t}\\
&+c\, \left[\left(\frac{1}{\sigma}\right)^\frac{N}{p-1}+\left(\frac{1}{\sigma}\right)^\frac{N}{p}\right]\,\left\{\tau\,\mathcal{E}(u;x_0,2\,r)+\tau^\frac{p-2}{p-1}\,\left(r^{s\,p-\frac{N}{q}}\,\|F\|_{L^{q}(B_{2\,r})}\right)^\frac{1}{p-1}\right\},
\end{split}
\end{equation}
where $c=c(N,s,p)>0$, $\alpha=\alpha(N,s,p)>0$ is smaller than $1$ and $0<\sigma,\tau<1$ are free parameters. The proof of \eqref{continua?} relies on a perturbative argument permitting to transfer the decay estimate of $v$ to $u$. As before, $v$ is the solution of the homogeneous problem in the relevant ball, having $u$ has boundary datum. Then \eqref{continua?} can be proved as inequality (5.5) in \cite{KMS}, by replacing $q_*$ there with $p$ and using Lemma \ref{lm:rimpiazzo} for $w=u-v$ in place of their Lemma 3.5.
\end{proof}

\begin{coro}[Eigenfunctions]
\label{coro:eigencontinue}
Let $1<p<\infty$ and $0<s<1$. Every $(s,p)-$eigenfunction of the open bounded set $\Omega\subset\mathbb{R}^N$ is continuous.
\end{coro}
\begin{proof}
If $s\,p>N$ there is nothing to prove. For $s\,p\le N$, we already know by Remark \ref{oss:eigenlimitate} that eigenfunctions are bounded. In particular an $(s,p)-$eigenfunction $u$ solves \eqref{puntuale} with $F=\lambda\, |u|^{p-2}\,u\in L^\infty(\Omega)$. Then the conclusion follows from Theorem \ref{teo:continue!}.
\end{proof}

\section{The second eigenvalue}
\label{sec:4}

We start by defining
\[
\lambda_2(\Omega)=\inf_{f\in \mathcal{C}_1(\Omega)} \max_{u\in\mathrm{Im}(f)} \|u\|^p_{\widetilde W^{s,p}_0(\Omega)},
\]
where the set $\mathcal{C}_1(\Omega)$ is given by
\[
\mathcal{C}_1(\Omega)=\left\{f:\mathbb{S}^1\to \mathcal{S}_p(\Omega)\, :\, f \mbox{ odd and continuous}\right\},
\]
and we recall that $\mathcal{S}_p(\Omega)$ is defined in \eqref{sfera}.
We have the following preliminary result.
\begin{teo}
\label{teo:basic}
The quantity $\lambda_2(\Omega)$ is an $(s,p)-$eigenvalue. Moreover we have $\lambda_1(\Omega)<\lambda_2(\Omega)$ and every eigenfunction $u\in\mathcal{S}_p(\Omega)$ associated to $\lambda_2(\Omega)$ has to change sign. 
\end{teo}
\begin{proof}
We divide the proof in three parts.
\vskip.2cm\noindent
\underline{\it Part 1: $\lambda_2(\Omega)$ is an eigenvalue.} In order to prove that $\lambda_2(\Omega)$ is a critical point of the functional $\Phi_{s,p}$ defined in \eqref{phi} on the manifold $\mathcal{S}_p(\Omega)$, it is sufficient to check that $\Phi_{s,p}$ verifies the Palais-Smale condition. The claim will then follow by applying \cite[Proposition 2.7]{Cu}. 
Let us take a sequence $\{u_n\}_{n\in\mathbb{N}}\subset\mathcal{S}_p(\Omega)$ such that
\begin{equation}
\label{PS}
\Phi_{s,p}(u_n)\le C\qquad \mbox{ and }\qquad \lim_{n\to\infty}\left\|D\Phi_{s,p}(u_n)_{|T_{u_n}\mathcal{S}_p(\Omega)}\right\|_*=0,
\end{equation}
where $D\Phi_{s,p}(u_n)$ denotes the differential of $\Phi_{s,p}$ at the point $u_n$ and $T_{u_n} \mathcal{S}_p(\Omega)$ is the tangent space to $\mathcal{S}_p(\Omega)$ at the point $u_n$, given by
\[
T_{u_n} \mathcal{S}_p(\Omega)=\left\{\varphi\in \widetilde W^{s,p}_0(\Omega)\, :\, \int_\Omega |u_n|^{p-2}\, u_n\, \varphi\, dx=0\right\}.
\]
Then the second hypothesis in \eqref{PS} implies that there exists a sequence $\varepsilon_n>0$ converging to $0$ and such that
\begin{equation}
\label{PS2}
\Big|D\Phi_{s,p}(u_n)[\varphi]\Big|\le \varepsilon_n\, \|\varphi\|_{\widetilde W^{s,p}_0(\Omega)},\qquad \mbox{ for every }\varphi\in T_{u_n} \mathcal{S}_p(\Omega).
\end{equation}
On the other hand, by the first hypothesis in \eqref{PS} we can infer that $\{u_n\}_{n\in\mathbb{N}}$ is converging to a function $u$ (up to a subsequence), strongly in $L^p(\Omega)$ and weakly in $\widetilde W^{s,p}_0(\Omega)$ (see for example \cite[Theorem 2.7]{BLP}). By strong convergence in $L^p(\Omega)$ we have of course $u\in \mathcal{S}_p(\Omega)$. Also observe that the sequence
\[ 
\delta_n := \int_\Omega |u_n|^{p-2}\, u_n\, u\, dx, 
\]
converges to $1$, as $n$ goes to $\infty$.
We then define the new sequence $\{v_n\}_{n\in\mathbb{N}}$ by
\[
v_n=\delta_n\, u_n-u,\qquad n\in\mathbb{N},
\]
and we observe that $v_n \in T_{u_n}\mathcal{S}_p(\Omega)$ for every $n$. 
Thus by \eqref{PS2} we get
\[
\lim_{n\to\infty}\Big|D\Phi_{s,p}(u_n)[u_n-u]\Big|\le \lim_{n\to\infty}\Big|D\Phi_{s,p}(u_n)[v_n]\Big|+\lim_{n\to\infty}|1-\delta_n|\, \Big|D\Phi_{s,p}(u_n)[u_n]\Big|=0,
\]
thanks to the fact that by homogeneity we have
\[
\Big|D\Phi_{s,p}(u_n)[u_n]\Big|=p\, |\Phi_{s,p}(u_n)|,
\]
which is uniformly bounded by hypothesis \eqref{PS}. On the other hand, by weak convergence of $u_n$ to $u$ we also have
\[
\lim_{n\to\infty}\Big|D\Phi_{s,p}(u)[u_n-u]\Big|=0.
\]
In conclusion we get
\begin{equation}
\label{azzero}
\lim_{n\to\infty} \Big|D\Phi_{s,p}(u_n)[u_n-u]-D\Phi_{s,p}(u)[u_n-u]\Big|=0.
\end{equation}
If we set
\[
U_n(x,y)=u_n(x)-u_n(y)\qquad \mbox{ and }\qquad U(x,y)=u(x)-u(y),
\]
then \eqref{azzero} can be rewritten as
\[
\lim_{n\to\infty} \int_{\mathbb{R}^N} \int_{\mathbb{R}^N} \frac{\Big(|U_n(x,y)|^{p-2}\, U_n(x,y)-|U(x,y)|^{p-2}\, U(x,y)\Big)\,\Big(U_n(x,y)-U(x,y)\Big)}{|x-y|^{N+s\,p}}\, dx\,dy=0.
\]
By appealing to \eqref{uno}, the previous implies for $p\ge 2$
\[
\lim_{n\to\infty} \int_{\mathbb{R}^N} \int_{\mathbb{R}^N} \frac{|U_n(x,y)-U(x,y)|^p}{|x-y|^{N+s\,p}}\, dx\,dy=0,
\]
so that $u_n$ strongly converges in $\widetilde W^{s,p}_0(\Omega)$  to $u$ (up to a subsequence). For the case $1<p< 2$, we use \eqref{due} raised to the power $p/2$, so that we get
\[
\begin{split}
\lim_{n\to\infty} \int_{\mathbb{R}^N}\int_{\mathbb{R}^N} &\frac{|U_n(x,y)-U(x,y)|^p}{|x-y|^{N+s\,p}}\,dx\,dy\\
&\le \left(\frac{1}{p-1}\right)^\frac{p}{2}\int_{\mathbb{R}^N}\int_{\mathbb{R}^N} \frac{\Big(|U_n(x,y)|^{p-2}\, U_n(x,y)-|U(x,y)|^{p-2}\, U(x,y)\Big)^\frac{p}{2}}{|x-y|^{N+s\,p}}\\
&\times \,\Big(U_n(x,y)-U(x,y)\Big)^\frac{p}{2}\,\Big(|U_n(x,y)|^2+|U(x,y)|^2\Big)^{\frac{2-p}{2}\,\frac{p}{2}}\, dx\,dy.
\end{split}
\]
It is now sufficient to use H\"{o}lder inequality with exponents $2/p$ and $2/(2-p)$ in order to infer again the strong convergence.
\vskip.2cm\noindent
\underline{\it Part 2: $\lambda_2(\Omega)>\lambda_1(\Omega)$.} We can use a simple topological argument, like in \cite[Theorem 4.2]{BF_nodea}. Let us argue by contradiction and suppose that
\[
\lambda_{2}(\Omega)=\inf_{f\in \mathcal{C}_1(\Omega)}\, \max_{u\in \mathrm{Im}(f)} \|u\|_{\widetilde W^{s,p}_0(\Omega)}^p=\lambda_1(\Omega),
\]
so that, for all $n\in\mathbb{N}$ there exists an odd continuous mapping $f_n:\mathbb{S}^1\to  \mathcal{S}_p(\Omega)$ such that
\begin{equation}\label{sfuggente}
  \max_{u\in f_n(\mathbb{S}^1)} \|u\|^p_{\widetilde W^{s,p}_0(\Omega)} \le \lambda_1(\Omega)+\frac{1}{n}.
\end{equation}
Let us denote by $u_1$ the unique (modulo the choice of the sign) solution of \eqref{problema}.
Let $0<\varepsilon\ll 1$ and consider the two neighborhoods
\[
\mathcal{B}_\varepsilon^+ = \{u \in \mathcal{S}_p(\Omega) : \|u-u_1\|_{L^p(\Omega)} < \varepsilon\}
\qquad\qquad
 \mathcal{B}_\varepsilon^- = \{u\in \mathcal{S}_p(\Omega) : \|u- (-u_1)\|_{L^p(\Omega)}<\varepsilon \},
\]
which are disjoint, by construction. 
Since the mapping $f_n$ is odd and continuous, for every $n\in\mathbb{N}$ the image $f_n(\mathbb{S}^{1})$ is symmetric and connected, then it
can not be contained in $\mathcal{B}_\varepsilon^+\cup \mathcal{B}_\varepsilon^-$, the latter being symmetric and disconnected. So we can pick an element
\begin{equation}
\label{fuori}
u_n \in f_n(\mathbb{S}^{1})\setminus \Big(\mathcal{B}_\varepsilon^+\cup \mathcal{B}_\varepsilon^- \Big).
\end{equation}
This yields a sequence $\{u_n\}_{n\in\mathbb{N}}\subset \mathcal{S}_p(\Omega)$, which is bounded in $\widetilde W^{s,p}_0(\Omega)$ by \eqref{sfuggente}.
Hence, there exists a function $v\in \mathcal{S}_p(\Omega)$ such that $\{u_n\}_{n\in\mathbb{N}}$ converges to $v$ weakly in $\widetilde W^{s,p}_0(\Omega)$
and strongly in $L^p(\Omega)$, possibly by passing to a subsequence. By the weak convergence it follows that
\[
\|v\|^p_{\widetilde W^{s,p}_0(\Omega)}
\le \liminf_{n\to\infty} \|u_n\|^p_{\widetilde W^{s,p}_0(\Omega)}=\lambda_1(\Omega).
\]
This in turn shows that $v\in \mathcal{S}_p(\Omega)$ is a global minimizer, so that
either $v=u_1$ or $v=-u_1$. On the other hand, by strong $L^p$ convergence we also have
\[
v\in\mathcal{S}_p(\Omega)\setminus\Big(\mathcal{B}_\varepsilon^+\cup \mathcal{B}_\varepsilon^- \Big).
\] 
This gives a contradiction and thus $\lambda_2(\Omega)>\lambda_1(\Omega)$.
\vskip.2cm\noindent
\underline{\it Part 3: $\lambda_2(\Omega)$ admits only sign-changing eigenfunctions.} This follows directly from the previous step and Theorem \ref{teo:primo!}.
\end{proof}
The following result justifies the notation we used for $\lambda_2(\Omega)$. Indeed, {\it the latter is exactly the second $(s,p)-$eigenvalue.}
\begin{prop}
\label{lm:secondo}
Let $0<s<1$ and $1<p<\infty$. Let $\Omega\subset\mathbb{R}^N$ be an open and bounded set. For every eigenvalue $\lambda>\lambda_1(\Omega)$ we have
\[
\lambda_2(\Omega)\le \lambda.
\]
In particular $\lambda_1(\Omega)$ is isolated.
\end{prop}
\begin{proof}
Let $(u,\lambda)$ be an eigenpair, with $\lambda>\lambda_1(\Omega)$. By Theorem \ref{teo:primo!}, $u$ has to change sign in $\Omega$, i.e. $u=u_+-u_-$ with $u_+\not\equiv 0$ and $u_-\not\equiv 0$. By testing the equation solved by $u$ against $u_+$ and $u_-$, we get
\[
\lambda\, \int_\Omega u_+^p\, dx=\int_{\mathbb{R}^N}\int_{\mathbb{R}^N} \frac{|u(x)-u(y)|^{p-2}\, (u(x)-u(y))}{|x-y|^{N+s\,p}}\, (u_+(x)-u_+(y))\, dx\,dy,
\]
and
\[
-\lambda\, \int_\Omega u_-^p\, dx=\int_{\mathbb{R}^N}\int_{\mathbb{R}^N} \frac{|u(x)-u(y)|^{p-2}\, (u(x)-u(y))}{|x-y|^{N+s\,p}}\, (u_-(x)-u_-(y))\, dx\,dy.
\]
Let us set for simplicity
\[
U(x,y)=u_+(x)-u_+(y)\qquad \mbox{ and }\qquad V(x,y)=u_-(x)-u_-(y),
\]
then 
\[
u(x)-u(y)=(u_+(x)-u_-(x))-(u_+(y)-u_-(y))=U(x,y)-V(x,y).
\]
In this way we can rewrite
\[
\lambda\, \int_\Omega u_+^p\, dx=\int_{\mathbb{R}^N}\int_{\mathbb{R}^N} \frac{|U-V|^{p-2}\, (U-V)}{|x-y|^{N+s\,p}}\,U\, dx\,dy,
\]
and
\[
-\lambda\, \int_\Omega u_-^p\, dx=\int_{\mathbb{R}^N}\int_{\mathbb{R}^N} \frac{|U-V|^{p-2}\, (U-V)}{|x-y|^{N+s\,p}}\,V\, dx\,dy.
\]
Let $(\omega_1,\omega_2)\in \mathbb{S}^1$, by multiplying the previous two identities by $|\omega_1|^p$ and $|\omega_2|^p$ and subtracting them, we can then arrive at
\begin{equation}
\label{schifo}
\lambda=\frac{\displaystyle\int_{\mathbb{R}^N}\int_{\mathbb{R}^N} \frac{\Big[|\omega_1|^p\, |U-V|^{p-2}\, (U-V)\, U-|\omega_2|^p\, |U-V|^{p-2}\, (U-V)\, V\Big]}{|x-y|^{N+s\,p}}\,dx\,dy}{\displaystyle|\omega_1|^p\, \int_\Omega u_+^p+|\omega_2|^p\, \int_\Omega u_-^p\, dx}.
\end{equation}
Observe that by homogeneity the integrand in the numerator can be written as
\[
|\omega_1\, U-\omega_1\, V|^{p-2}\,(\omega_1\, U-\omega_1\,V)\,\omega_1\,U-|\omega_2\, U-\omega_2\, V|^{p-2}\,(\omega_2\, U-\omega_2\,V)\,\omega_2\,V.
\]
We now claim that
\begin{equation}
\label{claim}
\begin{split}
|\omega_1\, U-\omega_1\, V|^{p-2}\,(\omega_1\, U-\omega_1\,V)\,\omega_1\,U&-|\omega_2\, U-\omega_2\, V|^{p-2}\,(\omega_2\, U-\omega_2\,V)\,\omega_2\,V\\
&\ge |\omega_1\, U-\omega_2\, V|^p.
\end{split}
\end{equation}
Let us now assume \eqref{claim} for the moment and show how the proof ends. Indeed, if we define the following element of $\mathcal{C}_1(\Omega)$
\[
f(\omega)=\frac{\omega_1\, u_+-\omega_2\, u_-}{\displaystyle\left(|\omega_1|^p\, \int_\Omega u_+^p+|\omega_2|^p\, \int_\Omega u_-^p\, dx\right)^{1/p}},\qquad \omega=(\omega_1,\omega_2)\in\mathbb{S}^1,
\]
then with the notations above we have
\[
\begin{split}
\|f(\omega)\|^p_{\widetilde W^{s,p}_0(\Omega)}
&=\frac{\displaystyle\int_{\mathbb{R}^N}\int_{\mathbb{R}^N}\frac{|\omega_1\, U-\omega_2\,V|^p}{|x-y|^{N+s\,p}}\, dx\,dy}{\displaystyle|\omega_1|^p\, \int_\Omega u_+^p+|\omega_2|^p\, \int_\Omega u_-^p\, dx}.
\end{split}
\]
If we now use the pointwise inequality \eqref{claim} and recall the relation \eqref{schifo} for $\lambda$, we get
\[
\|f(\omega)\|^p_{\widetilde W^{s,p}_0(\Omega)}\le \lambda,\qquad \mbox{ for every }\omega\in\mathbb{S}^1.
\]
By appealing to the definition of $\lambda_2(\Omega)$ we get the desired conclusion. Observe that $\lambda_1(\Omega)$ is isolated thanks to the fact that $\lambda_2(\Omega)>\lambda_1(\Omega)$, which follows from Theorem \ref{teo:basic}.
\vskip.2cm\noindent
In order to conclude the proof, let us now prove the vital pointwise inequality \eqref{claim}. We start observing that $U\,V \leq 0$. 
\par
If $\omega_1=0$, then the left-hand side of \eqref{claim} reduces to 
\[
-|\omega_2\, U-\omega_2\, V|^{p-2}\,(\omega_2\, U-\omega_2\,V)\,\omega_2\,V=-|\omega_2|^p\, |U-V|^{p-2}\,(U-V)\,V.
\]
If $V=0$, the inequality is obvious. If $V < 0$, then $U\, V\le 0$ implies $U\ge 0$ so that $U-V\ge -V$. By using the monotonicity of the map $\tau\mapsto|\tau|^{p-2}\, \tau$ we get
\[
-|\omega_2|^p\, |U-V|^{p-2}\,(U-V)\,V\ge -|\omega_2|^p\,|-V|^{p-2}\,(-V)\, V=|\omega_2\, V|^p,
\]
which coincides with the right-hand side of \eqref{claim} when $\omega_1=0$. If on the contrary we have $V>0$, then $U\le 0$ and in this case $U-V\le -V$. Using again the monotonicity of the map $\tau\mapsto|\tau|^{p-2}\, \tau$, we get again
\[
-|\omega_2|^p\, |U-V|^{p-2}\,(U-V)\,V\ge -|\omega_2|^p\,|-V|^{p-2}\,(-V)\, V=|\omega_2\, V|^p.
\]
This proves \eqref{claim} for the case $\omega_1=0$. 
\par
We now assume $\omega_1\not =0$, then by dividing everything by $|\omega_1|^p$ we get that \eqref{claim} is equivalent to prove
\[
\begin{split}
|U-V|^{p-2}\,(U-V)\,C&\ge\left|\frac{\omega_2}{\omega_1}\, U-\frac{\omega_2}{\omega_1}\, V\right|^{p-2}\,\left(\frac{\omega_2}{\omega_1}\, U-\frac{\omega_2}{\omega_1}\,V\right)\,\frac{\omega_2}{\omega_1}\,V+\left|U-\frac{\omega_2}{\omega_1}\, V\right|^p.
\end{split}
\]
We now observe that the previous inequality is a direct consequence of Lemma \ref{lm:salvaculo}, thus the proof is complete.
\end{proof}
\begin{oss}
The fact that $\lambda_1(\Omega)$ is isolated was also proved in \cite[Theorem 19]{LL}, under the restriction $s\,p>N$. The restriction was needed in order to have the eigenfunctions continuous, a fact that we would have now for free from Corollary \ref{coro:eigencontinue}. However, our proof is different and does not need continuity of eigenfunctions. 
\end{oss}
\begin{oss}
It is not difficult to see that $\lambda_2(\Omega)$ coincides with $\lambda_2$ defined in \cite[Section 2]{IS} by means of a cohomological index. We leave the details to the interested reader.
\end{oss}

\section{Mountain pass characterization}
\label{sec:5}

In this section we prove an alternative characterization of $\lambda_2(\Omega)$ as a mountain pass level. 
In order to prove such a characterization, the following technical result will be useful.
\begin{lm}
Let $1<p<\infty$ and $0<s<1$. Let $\Omega\subset\mathbb{R}^N$ be an open and bounded set. For every $u\in\mathcal{S}_p(\Omega)$ we set
\[
U(x,y)=u_+(x)-u_+(y)\qquad \mbox{ and }\qquad V(x,y)=u_-(x)-u_-(y),
\]
and we define the continuous curve on $\mathcal{S}_p(\Omega)$
\[
\gamma_t=\frac{u_+-\cos(\pi\,t)\,u_-}{\|u_+-\cos(\pi\,t)\,u_-\|_{L^p(\Omega)}},\qquad t\in \left[0,\frac{1}{2}\right].
\]
Let us suppose that we have
\begin{equation}
\label{allacazzo}
\begin{split}
\|u_-\|^p_{L^p(\Omega)}\, &\int_{\mathbb{R}^N}\int_{\mathbb{R}^N} \frac{|U-V|^{p-2}\,(U-V)\,U}{|x-y|^{N+s\,p}}\,dx\, dy\\
&+ \|u_+\|^p_{L^p(\Omega)}\,\int_{\mathbb{R}^N}\int_{\mathbb{R}^N} \frac{|U-V|^{p-2}\,(U-V)\,V}{|x-y|^{N+s\,p}}\,dx\, dy\le 0,
\end{split}
\end{equation}
then there holds
\[
\|\gamma_t\|^p_{\widetilde W^{s,p}_0(\Omega)}\le \|u\|^p_{\widetilde W^{s,p}_0(\Omega)},\qquad t\in\left[0,\frac{1}{2}\right].
\]
\end{lm}
\begin{proof}
We start by observing that
\[
\begin{split}
\|\gamma_t\|^p_{\widetilde W^{s,p}_0(\Omega)}
&=\frac{\displaystyle\int_{\mathbb{R}^N}\int_{\mathbb{R}^N} \frac{\big|U-\cos(\pi\,t)\,V\big|^p}{|x-y|^{N+s\,p}}\, dx\,dy}{\displaystyle \int_\Omega u_+^p\, dx+|\cos(\pi\,t)|^p\,\int_\Omega u_-^p\, dx},\qquad t\in \left[0,\frac{1}{2}\right].
\end{split}
\]
By definition we have $U\cdot V\le 0$, then by using Lemma \ref{lm:salvaculo} we get
\[
\begin{split}
\|\gamma_t\|^p_{\widetilde W^{s,p}_0(\Omega)}&\le \frac{\displaystyle\int_{\mathbb{R}^N}\int_{\mathbb{R}^N} \frac{\big|U-V\big|^{p-2}\,(U-V)\, U}{|x-y|^{N+s\,p}}\, dx\,dy-|\cos(\pi\,t)|^p\,\int_{\mathbb{R}^N}\int_{\mathbb{R}^N} \frac{\big|U-V\big|^{p-2}\,(U-V)\, V}{|x-y|^{N+s\,p}}\,dx\,dy}{\displaystyle \int_\Omega u_+^p\, dx+|\cos(\pi\,t)|^p\,\int_\Omega u_-^p\, dx},
\end{split}
\]
for every $t\in[0,1/2]$.
Observe that the term on the right-hand side has the form
\[
\frac{a-s\,b}{c+s\,d},\qquad \mbox{ for }s\in[0,1],
\]
with $a,b\in\mathbb{R}$ and $c,d\ge 0$ such that $c+d>0$.
In order to get the conclusion, it is then sufficient to observe that the function
\[
s\mapsto \frac{a-s\,b}{c+s\,d},
\]
is monotone increasing if and only if
\[
c\,b+d\,a\le 0,
\]
that is if and only if
\[
\begin{split}
\left(\int_\Omega u_-^p\, dx\right)\, &\int_{\mathbb{R}^N}\int_{\mathbb{R}^N} \frac{|U-V|^{p-2}\,(U-V)\,U}{|x-y|^{N+s\,p}}\,dx\, dy\\
&+ \left(\int_\Omega u_+^p\, dx\right)\,\int_{\mathbb{R}^N}\int_{\mathbb{R}^N} \frac{|U-V|^{p-2}\,(U-V)\,V}{|x-y|^{N+s\,p}}\,dx\, dy\le 0,
\end{split}
\]
which is exactly hypothesis \eqref{allacazzo}. By using this and recalling that $u$ has unit $L^p$ norm, we get for $t\in[0,1/2]$
\[
\begin{split}
\|\gamma_t\|^p_{\widetilde W^{s,p}_0(\Omega)}&\le \int_{\mathbb{R}^N}\int_{\mathbb{R}^N} \left[\frac{\big|U-V\big|^{p-2}\,(U-V)\, U}{|x-y|^{N+s\,p}}-\frac{\big|U-V\big|^{p-2}\,(U-V)\, V}{|x-y|^{N+s\,p}}\right]\, dx\,dy=\|u\|^p_{\widetilde W^{s,p}_0(\Omega)},
\end{split}
\]
as desired.
\end{proof}
The following simple remark will be important.
\begin{oss}
\label{oss:meno}
Let $u\in\mathcal{S}_p(\Omega)$ be a function that does not satisfy \eqref{allacazzo}, i.e.
\[
\begin{split}
\|u_-\|^p_{L^p(\Omega)}\, &\int_{\mathbb{R}^N}\int_{\mathbb{R}^N} \frac{|U-V|^{p-2}\,(U-V)\,U}{|x-y|^{N+s\,p}}\,dx\, dy\\
&+ \|u_+\|^p_{L^p(\Omega)}\,\int_{\mathbb{R}^N}\int_{\mathbb{R}^N} \frac{|U-V|^{p-2}\,(U-V)\,V}{|x-y|^{N+s\,p}}\,dx\,dy> 0.
\end{split}
\] 
Then it is easy to see that the function $v=-u\in\mathcal{S}_p(\Omega)$ satisfies \eqref{allacazzo}.
\end{oss}
Let us define
\[
\Gamma(u_1,-u_1)=\{\gamma\in C^0([0,1];\mathcal{S}_p(\Omega))\, :\, \gamma_0=u_1,\, \gamma_1=-u_1\},
\]
the set of continuous curves on $\mathcal{S}_p(\Omega)$ connecting the two solutions $u_1$ and $-u_1$ of \eqref{problema}. We have the following characterization for $\lambda_2(\Omega)$. The proof is similar to that of \cite[Proposition 5.4]{BF_nodea}.
\begin{teo}[Mountain pass characterization]
\label{teo:dolomiten}
Let $1<p<\infty$ and $0<s<1$. Let $\Omega\subset\mathbb{R}^N$ be an open and bounded set, then we have
\[
\lambda_2(\Omega)=\inf_{\gamma\in \Gamma(u_1,-u_1)}\,\max_{u\in \mathrm{Im}(\gamma)} \|u\|^p_{\widetilde W^{s,p}_0(\Omega)}.
\]
\end{teo}
\begin{proof}
We first observe that the inequality 
\[
\lambda_2(\Omega)\le \inf_{\gamma\in \Gamma(u_1,-u_1)}\,\max_{u\in \gamma} \|u\|^p_{\widetilde W^{s,p}_0(\Omega)},
\]
is easily seen to hold. Indeed, given $\gamma\in \Gamma(u_1,-u_1)$, we can consider the closed loop $\{\gamma\}\cup\{-\gamma\}$ and identify it with an element of $\mathcal{C}_1(\Omega)$. 
\vskip.2cm\noindent
Let us now prove the converse. For every $n\in\mathbb{N}$, we pick $f_n\in\mathcal{C}_1(\Omega)$ ``almost optimal'', i.e. such that
\[
\max_{u\in\mathrm{Im}(f_n)} \|u\|^p_{\widetilde W^{s,p}_0(\Omega)}\le \lambda_2(\Omega)+\frac{1}{n}.
\] 
Since $f_n$ is odd, the set $\mathrm{Im}(f_n)$ is symmetric with respect to the origin. Then thanks to Remark \ref{oss:meno} there exists $u_n\in\mathrm{Im}(f_n)$ which verifies hypothesis \eqref{allacazzo}. This implies that on the curve
\[
\gamma_{n,t}=\frac{(u_n)_+-\cos(\pi\,t)\,(u_n)_-}{\|(u_n)_+-\cos(\pi\,t)\,(u_n)_-\|_{L^p(\Omega)}},\qquad 0\le t\le \frac{1}{2},
\]
we have
\begin{equation}
\label{mezzo}
\|\gamma_{n,t}\|_{\widetilde W^{s,p}_p(\Omega)}^p\le \lambda_2(\Omega)+\frac{1}{n},\qquad 0\le t\le \frac{1}{2}.
\end{equation}
Observe that the curve $\gamma_n$ is connecting $u_n$ to its (renormalized in $L^p$) positive part, without increasing the energy. We now in turn connect the function $(u_n)_+/\|(u_n)_+\|_{L^p(\Omega)}$ to the first eigenfunction $u_1$: at this aim, we recall that on the curve
\[
\sigma_{n,t}=\left((1-t)\, \frac{(u_n)_+^p}{\|(u_n)_+\|_{L^p(\Omega)}}+t\, u_1^p\right)^\frac{1}{p},\qquad t\in[0,1],
\]
{\it our energy functional is convex} (see \cite[Lemma 4.1]{FP} or also \cite[Proposition 4.1]{BF}), i.e. 
\[
\|\sigma_{n,t}\|^p_{\widetilde W^{s,p}_0(\Omega)}\le (1-t)\, \frac{\|(u_n)_+\|^p_{\widetilde W^{s,p}_0(\Omega)}}{\|(u_n)_+\|^p_{L^p(\Omega)}}+t\, \|u_1\|^p_{\widetilde W^{s,p}_0(\Omega)}.
\]
Thus in particular from \eqref{mezzo} we have
\[
\|\sigma_{n,t}\|^p_{\widetilde W^{s,p}_0(\Omega)}\le \lambda_2(\Omega)+\frac{1}{n},\qquad t\in[0,1].
\]
Thus we can glue together $\gamma_n$ and $\sigma_n$ and build the new curve
\[
\widetilde \gamma_{n,t}=\left\{\begin{array}{cc}
\gamma_{n,t},& t\in[0,1/2],\\
\sigma_{n,(2\,t-1)},& t\in[1/2,1]
\end{array}
\right.
\]
which is connecting $u_n$ to $u_1$ and on which the energy is always less that $\lambda_2(\Omega)+1/n$. Finally, if we glue together $\widetilde \gamma_{n}$, $-\widetilde \gamma_{n}$ and $f_n$, perform a suitable reparameterization and use the fact that the energy functional is even, we obtain a new continuous curve $\Sigma_n\in\Gamma(u_1,-u_1)$ such that
\[
\max_{t\in[0,1]}\|\Sigma_{n,t}\|^p_{\widetilde W^{s,p}_0(\Omega)}\le \lambda_2(\Omega)+\frac{1}{n},\qquad n\in\mathbb{N}.
\]
This of course implies that
\[
\inf_{\gamma\in \Gamma(u_1,-u_1)}\,\max_{u\in \mathrm{Im}(\gamma)} \|u\|^p_{\widetilde W^{s,p}_0(\Omega)}\le \lambda_2(\Omega)+\frac{1}{n}.
\]
By taking the limit as $n$ goes to $\infty$, we finally obtain the desired conclusion.
\end{proof}
\begin{oss}
In the nonlocal case, the previous characterization has been proven for $p=2$ in \cite{GS} for the so-called {\it Fu\v{c}ik spectrum}, by adapting the proof of \cite{CDG2}.
\end{oss}

\section{A sharp lower bound}
\label{sec:6}

We are going to prove a sharp lower bound on $\lambda_2(\Omega)$ in terms of the measure of $\Omega$. The following simple result will be important.
\begin{lm}[Nodal domains]
\label{lm:LL}
Let $\lambda>\lambda_1(\Omega)$ be an $(s,p)-$eigenvalue. Let $u\in\mathcal{S}_p(\Omega)$ be an associated eigenfunction and set
\[
\Omega_+=\{x\in\Omega\, :\, u(x)>0\}\qquad \mbox{ and }\qquad \Omega_-=\{x\in\Omega\, :\, u(x)<0\}.
\]
Then we have
\[
\lambda>\max\{\lambda_1(\Omega_+),\, \lambda_1(\Omega_-)\}.
\]
\end{lm}
\begin{proof}
We first observe that since $(s,p)-$eigenfunctions are continuous by Corollary \ref{coro:eigencontinue}, the sets $\Omega_+$ and $\Omega_-$ are open. Thus $\lambda_1(\Omega_+)$ and $\lambda_1(\Omega_-)$ are well-defined.
\par
The case $p\ge 2$ is already contained in \cite[Theorem 17]{LL}, so let us focus on the case $1<p<2$.
Since $u$ is sign-changing, we can write $u=u_+-u_-$, where $u_+$ and $u_-$ are the positive and negative parts respectively. By testing the equation solved by $u$ against $u_+$ we get
\[
\lambda\,\int_\Omega |u_+|^p\, dx=\int_{\mathbb{R}^N}\int_{\mathbb{R}^N} \frac{|u(x)-u(y)|^{p-2}\, (u(x)-u(y))}{|x-y|^{N+s\,p}}\, (u_+(x)-u_+(y))\,dx\,dy.
\]
Then we can apply Lemma \ref{lm:boh} with the choices
\[
a=u_+(x)-u_+(y)\qquad \mbox{ and }\qquad b=u_-(x)-u_-(y),
\]
and obtain
\[
\begin{split}
\lambda\,
\int_{\Omega} |u_+|^p\, dx &>\int_{\mathbb{R}^N}\int_{\mathbb{R}^N} \frac{|u_+(x)-u_+(y)|^p}{|x-y|^{N+s\,p}}\, dx\,dy.
\end{split}
\]
Since $u_+$ is admissible for the variational problem defining $\lambda_1(\Omega_+)$, we get  $\lambda> \lambda_1(\Omega_+)$.
\par
For the other set $\Omega_-$, we proceed similarly by testing the equation against $u_-$, thus getting
\[
\lambda\,\int_\Omega |u_-|^p\, dx=\int_{\mathbb{R}^N}\int_{\mathbb{R}^N} \frac{|u(y)-u(x)|^{p-2}\, (u(y)-u(x))}{|x-y|^{N+s\,p}}\, (u_-(x)-u_-(y))\,dx\,dy.
\]
If we now use again Lemma \ref{lm:boh}, this time with the choices
\[
a=u_-(x)-u_-(y)\qquad \mbox{ and }\qquad b=u_+(x)-u_+(y),
\]
we get also the estimate $\lambda> \lambda_1(\Omega_-)$.
This concludes the proof.
\end{proof}
The following is the main result of this section. This is the nonlocal version of the so-called {\it Hong-Krahn-Szego inequality} (see \cite[Theorem 3.2]{BF_MM}), which in the local case asserts that the second eigenvalue of the Dirichlet-Laplacian is minimized by the disjoint union of two equal balls, among sets of given measure.
\begin{teo}[Nonlocal Hong-Krahn-Szego inequality]
Let $0<s<1$ and $1<p<\infty$. For every $\Omega\subset\mathbb{R}^N$ open and bounded set we have
\begin{equation}
\label{HKS}
\lambda_2(\Omega)>\lambda_1(B),
\end{equation}
where $B$ is any $N-$dimensional ball such that $|B|=|\Omega|/2$. Equality is never attained in \eqref{HKS}, but the estimate is sharp in the following sense: if $\{x_n\}_{n\in\mathbb{N}}$, $\{y_n\}_{n\in\mathbb{N}} \subset \mathbb{R}^N$ are such that  
\[
\lim_{n\to\infty} |x_n-y_n|=+\infty,
\]
and we define \( 
\Omega_n := B_R(x_n) \cup B_R(y_n)\),
then
\[
\lim_{n\to\infty} \lambda_2(\Omega_n)= \lambda_1(B_R).
\]
\end{teo}
\begin{proof}
We divide the proof in two steps: at first we prove \eqref{HKS}, then we prove its sharpness.
\vskip.2cm\noindent
\underline{\it Inequality.} Let $u\in \mathcal{S}_p(\Omega)$ be an eigenfunction associated to $\lambda_2(\Omega)$. By Theorem \ref{teo:basic}, we know that $u$ is sign-changing, thus we define 
\[
\Omega_+ := \{ x \in \Omega\,:\,u(x)>0\}\qquad \mbox{ and }\qquad \Omega_- := \{ x \in \Omega\,:\,u(x)<0\}.
\]
By Lemma \ref{lm:LL} and the {\it nonlocal Faber-Krahn inequality} (see \cite[Theorem 3.5]{BLP}), we have
\[ 
\lambda_2(\Omega) > \lambda_1(\Omega_+) \geq \lambda_1(B_{R_1})\qquad \mbox{ and }\qquad \lambda_2(\Omega) > \lambda_1(\Omega_-) \geq \lambda_1(B_{R_2}). 
\]
where $B_{R_1}$ and $B_{R_2}$ are such that $|B_{R_1}|=|\Omega_+|$ and $|B_{R_2}|=|\Omega_-|$. Thus
\begin{equation}
\label{sconnessi}
\lambda_2(\Omega) > \max \{\lambda_1(B_{R_1}), \lambda_1(B_{R_2})\}.
\end{equation}
By the scaling properties of $\lambda_1$ we have $\lambda_1(B_R) = R^{-s\,p}\lambda_1(B_1)$, moreover we have the constraint
\[
|B_{R_1}|+|B_{R_2}|=|\Omega_+|+|\Omega_-|\le |\Omega|.
\]
Then it is easy to see that the right-hand side of \eqref{sconnessi} is minimal when $|B_{R_1}|=|B_{R_2}|=|\Omega|/2$, which implies the desired estimate \eqref{HKS}.
\vskip.2cm\noindent 
\underline{\it Sharpness.} In order to prove the second part of the claim, we define
\[ \Omega_n := B_R(x_n) \cup B_R(y_n),\]
where $\{x_n\}_{n\in\mathbb{N}}$, $\{y_n\}_{n\in\mathbb{N}}\subset \mathbb{R}^N$ are such that $|x_n-y_n|$ diverges as $n$ goes to $\infty$. Thus we can suppose that the two balls are disjoint. Let $u$ and $v$ be the positive normalized first eigenfunctions on $B_R(x_n)$ and $B_R(y_n)$ respectively (observe that their shape does not depend on the center of the ball), then we set for simplicity
\[
a(x,y)=u(x)-u(y)\qquad \mbox{ and }\qquad b(x,y)=v(x)-v(y).
\]
By Lemma \ref{lm:monotoni3} and the definition of $\lambda_2(\Omega)$ we have\footnote{For simplicity, we used the change of variable $(\omega_1,\omega_2)\mapsto \left(|\omega_1|^\frac{2-p}{p}\,\omega_1,|\omega_2|^\frac{2-p}{p}\,\omega_2\right)$.}
\[
\begin{split} \lambda_2(\Omega_n) &\leq \max_{|\omega_1|^p+|\omega_2|^p=1} \int_{\mathbb{R}^N}\int_{\mathbb{R}^N}\frac{|\omega_1\, a-\omega_2\,b|^p}{|x-y|^{N+s\,p}}\, dx\,dy \\ 
&\leq \max_{|\omega_1|^p+|\omega_2|^p=1}\left[ \int_{\mathbb{R}^N}\int_{\mathbb{R}^N}\frac{|\omega_1|^p\,|a|^p}{|x-y|^{N+s\,p}}\, dx\,dy + \int_{\mathbb{R}^N}\int_{\mathbb{R}^N}\frac{|\omega_2|^p\, |b|^p}{|x-y|^{N+s\,p}}\, dx\,dy  \right.\\ & \left. +c_p\, \int_{\mathbb{R}^N}\int_{\mathbb{R}^N}\frac{(|\omega_1\, a|^2+|\omega_2\, b|^2)^{\frac{p-2}{2}}\,|\omega_1\, \omega_2\,a\,b|}{|x-y|^{N+s\,p}}\, dx\,dy \right] \\ 
&=\lambda_1(B_R) + c_p\, \max_{|\omega_1|^p+|\omega_2|^p=1} \int_{\mathbb{R}^N}\int_{\mathbb{R}^N}\frac{\left(|\omega_1\, a|^2+|\omega_2\, b|^2\right)^{\frac{p-2}{2}}\,|\omega_1 \,\omega_2 \, a\,b|}{|x-y|^{N+s\,p}}\, dx\,dy.  
\end{split}
\]
Observe that, since
\[ 
a\,b = -u(x)\,v(y) - u(y)\,v(x),
\] 
the numerator in the last integral is nonzero only if $(x,y) \in B_R(x_n)\times B_R(y_n)$ or $(x,y) \in B_R(y_n)\times B_R(x_n)$. We set
\[ 
\mathcal{R}:= 2\,\max_{|\omega_1|^p+|\omega_2|^p=1} \int_{B_R(x_n)}\int_{B_R(y_n)}\left(|\omega_1\, a|^2+|\omega_2\, b|^2\right)^{\frac{p-2}{2}}\,|\omega_1 \,\omega_2 \, a\,b|\, dx\,dy<\infty,
\]
therefore we have
\[ 
\lim_{n\to\infty}\lambda_2(\Omega_n)  \leq \lambda_1(B_R) + \lim_{n\to\infty}\frac{c_p\,\mathcal{R}}{(|x_n-y_n|-2\,R)^{N+s\,p}}=\lambda_1(B_R).
\]
This concludes the proof.
\end{proof}

\begin{oss}
The previous result can be stated in scaling invariant form as follows
\begin{equation}
\label{HKSscaling}
\lambda_2(\Omega)>\left(\frac{2\,|B|}{|\Omega|}\right)^\frac{s\,p}{N}\, \lambda_1(B),
\end{equation}
where $B$ is any $N-$dimensional ball. 
\end{oss}

\appendix

\section{Some useful inequalities I}
\label{sec:A}

We will repeatedly use that for $1<p<\infty$ the real function
\[
J_p(t):=|t|^{p-2}\, t,
\]
is monotone increasing. 
\begin{lm}[Towards subsolutions]
Let $1<p<\infty$ and $f:\mathbb{R}\to\mathbb{R}$ be a $C^1$ convex function. For $\tau\ge 0$ we set
\[
J_{p,\tau}(t)=\left(\tau+|t|^2\right)^\frac{p-2}{2}\, t,\qquad t\in\mathbb{R},
\] 
then
\begin{equation}
\label{tassello}
\begin{split}
J_p(a-b)\, &\Big[A\,J_{p,\tau}(f'(a))-B\,J_{p,\tau}(f'(b))\Big]\\
&\ge \Big(\tau\,(a-b)^2+(f(a)-f(b))^2\Big)^\frac{p-2}{2}\,(f(a)-f(b))\, (A-B),
\end{split}
\end{equation}
for every $a,b\in\mathbb{R}$ and every $A,B\ge 0$.
\end{lm}
\begin{proof}
We can assume $a\not=b$, otherwise there is nothing to prove. By convexity of $f$ we have
\begin{equation}
\label{cacatella}
f(a)-f(b)\le f'(a)\, (a-b)\qquad \mbox{ and }\qquad  f(a)-f(b)\ge f'(b)\,(a-b).
\end{equation}
By writing the left-hand side of \eqref{tassello} as
\[
\begin{split}
J_p(a-b)&\Big[A\,J_{p,\tau}(f'(a))-B\,J_{p,\tau}(f'(b))\Big]\\
&=\Big(\tau\,(a-b)^2+\big(f'(a)\,(a-b)\big)^2\Big)^\frac{p-2}{2}\,f'(a)\,(a-b)\,A\\
&-\Big(\tau\,(a-b)^2+\big(f'(b)\,(a-b)\big)^2\Big)^\frac{p-2}{2}\,f'(b)\,(a-b)\,B,
\end{split}
\]
we can get the conclusion by simply using \eqref{cacatella} and the monotonicity of the function
\[
t\mapsto \Big(\tau\,(a-b)^2+t^2\Big)^\frac{p-2}{2}\,t,
\]
which is in turn the derivative of the convex function
\[
t\mapsto \frac{1}{p}\,\Big(\tau\,(a-b)^2+t^2\Big)^\frac{p}{2}.
\]
This gives the conclusion.
\end{proof}
The next pointwise inequality generalizes a similar estimate in \cite[Appendix C]{BLP}. 
\begin{lm}[Towards Moser's iteration]
Let $1<p<\infty$ and $g:\mathbb{R}\to\mathbb{R}$ be an increasing function. We define
\[
G(t)=\int_0^t g'(\tau)^\frac{1}{p}\, d\tau,\qquad t\in\mathbb{R},
\] 
then we have
\begin{equation}
\label{tassello2}
J_p(a-b)\, \big(g(a)-g(b)\big)\ge |G(a)-G(b)|^p.
\end{equation}
\end{lm}
\begin{proof}
We first observe that we can suppose $a>b$ without loss of generality. Then 
\[
\begin{split}
J_p(a-b)\, (g(a)-g(b))&=(a-b)^{p-1}\, \int_b^a g'(\tau)\, d\tau=(a-b)^{p-1}\, \int_b^a G'(\tau)^p\, d\tau\ge \left(\int_b^a G'(\tau)\, d\tau\right)^p,
\end{split}
\]
thanks to Jensen inequality.
\end{proof}
\begin{oss}
With the same proof one can show that if $g$ is decreasing, then
\[
|a-b|^{p-2}\, (a-b)\, \big(g(b)-g(a)\big)\ge |H(a)-H(b)|^p,\qquad \mbox{ where }\quad H(t)=\int_0^t {-g'(\tau)}^\frac{1}{p}\, d\tau.
\] 
\end{oss}
\begin{lm}
Let $\beta\ge 1$, then for every $a,b\ge 0$ we have
\begin{equation}
\label{rottolepalle}
|a-b|^{p}\,\left(a^{\beta-1}+b^{\beta-1}\right)\le \Big(\max\{1,(3-\beta)\}\Big)\, |a-b|^{p-2}\, (a-b)\, (a^\beta-b^\beta).
\end{equation}
\end{lm}
\begin{proof}
We first observe that \eqref{rottolepalle} is trivially true for $a=b$, thus let us consider $a\not =b$. It is not restrictive to assume that $a>b$, then \eqref{rottolepalle} is equivalent to
\[
(1-t)^{p}\,\left(1+t^{\beta-1}\right)\le C\, (1-t)^{p-1}\, (1-t^\beta),\qquad \mbox{ for } 0\le t<1,
\]
that is
\begin{equation}
\label{rottolepalle2}
(1-t)\,\left(1+t^{\beta-1}\right)\le C\,(1-t^\beta).
\end{equation}
By observing that
\[
(1-t)\,\left(1+t^{\beta-1}\right)=(1-t^\beta)+t^{\beta-1}-t,
\]
and remembering that $0\le t<1$, we easily get the conclusion for $\beta \ge 2$, since $t^{\beta-1}-t\le 0$ in this case. If on the contrary $1<\beta<2$, then by concavity of the function $\tau\mapsto\tau^{\beta-1}$ we have
\[
\begin{split}
t^{\beta-1}-t=(t^{\beta-1}-1)-(t-1)\le (\beta-1)\,(t-1)-(t-1)
&\le (2-\beta)\,(1-t^\beta).
\end{split}
\]
This finally shows \eqref{rottolepalle} for $1<\beta<2$ as well. The case $\beta=1$ is evident.
\end{proof}

\begin{lm} 
\label{betaandpi}
Let $p \geq 1$, then
\[ 
\left(\frac{1}{\beta}\right)^\frac{1}{p}\,\frac{\beta+p-1}{p} \geq 1,\qquad \mbox{ for every }\beta>0.
\]
\end{lm}
\begin{proof}
For $p=1$ there is nothing to prove, thus let us assume that $p>1$. The result follows from the convexity of the function $t\mapsto t^p$, which implies
\[
\beta-1\ge p\,(\beta^\frac{1}{p}-1).
\] 
By adding $p$ on both sides, we get the conclusion. 
\end{proof}

\section{Some useful inequalities II}
\label{sec:B}

We still use the notation $J_p(t)=|t|^{p-2}\, t$.

\begin{lm}
\label{lm:salvaculo}
Let $1<p<\infty$ and $U,V\in\mathbb{R}$ such that $U\, V\le 0$. We define the following function
\[
g(t)=|U-t\,V|^p+J_p(U-V)\,V\, |t|^p,\qquad t\in\mathbb{R}.
\]
Then we have
\[
g(t)\le g(1)=J_p(U-V)\,U,\qquad t\in\mathbb{R}.
\]
\end{lm}
\begin{proof}
Let us start observing that if $U=0$, then we have
\[
g(t)=0,\qquad \mbox{ for every } t\in\mathbb{R}.
\]
In the same manner, if $V=0$, then we have
\[
g(t)=|U|^p,\qquad \mbox{ for every }t\in\mathbb{R}.
\]
In both cases, the conclusion trivially holds true.
\par
Thus we can suppose that $U\, V\not =0$. Then we have
\[
\begin{split}
g'(t)&=-p\,J_p(U-t\,V)\,V+p\,J_p(U-V)\,V\, J_p(t)
=p\, V\, \Big[J_p(t\,U-t\,V)-J_p(U-t\,V)\Big].
\end{split}
\]
We distinguish two cases.
\begin{enumerate}
\item {\it Case $V<0$ and $U>0$}: then we have
\[
g'(t)\ge 0 \qquad \Longleftrightarrow\qquad t\,(U-V)\le U-t\,V\qquad \Longleftrightarrow \qquad t\le 1.
\]
This implies that $t=1$ is a global maximum point for the function $g$.
\vskip.2cm\noindent
\item {\it Case $V>0$ and $U<0$}: we now have 
\[
g'(t)\ge 0 \qquad \Longleftrightarrow\qquad t\,(U-V)\ge U-t\,V\qquad \Longleftrightarrow \qquad t\le 1,
\]
since now $U< 0$. Again, we get that $t=1$ is a global maximum point.
\end{enumerate}
In both cases, we get the desired conclusion.
\end{proof}

\begin{lm}
\label{lm:boh}
Let $1<p<\infty$, then for every $a,b\in\mathbb{R}$ such that $a\, b\le 0$, we have
\begin{equation}
\label{domininodali}
J_p(a-b)\, a\ge\left\{\begin{array}{lr}
|a|^p-(p-1)\,|a-b|^{p-2}\, a\,b,& \mbox{ if } 1<p\le 2,\\
&\\
|a|^p-(p-1)\, |a|^{p-2}\, a\, b, & \mbox{ if } p>2.
\end{array}.
\right.
\end{equation}
\end{lm}
\begin{proof}
We start with some elementary considerations. First of all, there is no loss of generality in supposing $a\ge 0$ and $b \le 0$.
Then we notice that the function $J_p$ on $[0,\infty)$ is convex for $p>2$ and concave for $1<p\le 2$. Thus
\begin{equation}
\label{uff1}
J_p(x)+J'_p(y)\,(y-x)\le J_p(y),\qquad \mbox{ for } 1<p\le 2,\quad 0\le x\le y,   
\end{equation}
and
\begin{equation}
\label{uff2}
J_p(x)+J'_p(x)\,(y-x)\le J_p(y),\qquad \mbox{ for } p> 2,\quad 0\le x\le y.
\end{equation}
We now come to the proof of \eqref{domininodali}, starting with the case $1<p\le 2$.
By using \eqref{uff1} with the choices 
\[
y=a-b\qquad \mbox{ and }\qquad x=a,
\]
we get
\[
\begin{split}
J_p(a-b)&\ge J_p(a)-J'_p(a-b)\,b=|a|^{p-2}\,a-(p-1)\, |a-b|^{p-2}\,b.
\end{split}
\]
and multiplying by $a\ge 0$ we conclude.
This ends the proof in the case $1<p\le 2$.
\vskip.2cm\noindent
The case $p>2$ is handled in a similar manner, by using \eqref{uff2} instead of \eqref{uff1}.
\end{proof}

\begin{lm} 
\label{lm:monotoni3}
Let $1<p<\infty$. Then there exists $c_p > 0$ such that for every $a,b\in\mathbb{R}$ we have
\begin{equation}
\label{pp}
|a-b|^p\le |a|^p+|b|^p+c_p\, (|a|^2+|b|^2)^\frac{p-2}{2}\, |a\, b|.
\end{equation}
\end{lm}
\begin{proof}
We first suppose $a\,b\ge 0$, without loss of generality we can suppose that $a,b\ge 0$ and $a\ge b$. Then we have
\[
|a-b|^p=(a-b)^p\le a^p\le |a|^p+|b|^p,
\]
thus \eqref{pp} is proved.
\par
Let us consider now the case $a\,b\le 0$. Without loss of generality, we can suppose that $a\ge 0$ and $b\le 0$. Then \eqref{pp} is equivalent to
\[
(a+\tau)^p\le a^p+\tau^p+c_p\,(a^2+\tau^2)^\frac{p-2}{2}\, a\,\tau,\qquad a,\tau\ge0.
\]
Of course this is easily seen to be true if $a=0$, so let us take $a>0$ and divide the previous by $a^p$. Then we are reduced to show that
\[
(1+m)^p\le 1+m^p+c_p\,\left(1+m^2\right)^\frac{p-2}{2}\, m,
\]
that is
\[
\sup_{m>0} \frac{(1+m)^p-1-m^p}{m\,\left(1+m^2\right)^\frac{p-2}{2}}<+\infty.
\]
To this end, it is sufficient to observe that
\[
\lim_{m\to 0^+}\frac{(1+m)^p-1-m^p}{m\,\left(1+m^2\right)^\frac{p-2}{2}}=p\qquad \mbox{ and }\qquad \lim_{m\to+\infty} \frac{(1+m)^p-1-m^p}{m\,\left(1+m^2\right)^\frac{p-2}{2}}=p.
\]
This concludes the proof.
\end{proof}
Finally, we recall the following classical inequalities. For the proofs the reader is referred to \cite[Section 10]{Li}.
\begin{lm}
\label{lm:monotoni1}
For $1<p\le 2$, we have
\begin{equation}
\label{due}
(|b|^2+|a|^2)^\frac{2-p}{2}\,(J_p(b)-J_p(a))\,(b-a)\ge
(p-1)\, |b-a|^2,\qquad a,b\in\mathbb{R}.
\end{equation}
For $2< p<\infty$ we have
 \begin{equation}
\label{uno}
(J_p(b)-J_p(a))\,(b-a)\ge
2^{2-p}\, |b-a|^p,\qquad  a,b\in\mathbb{R}.
\end{equation}
\end{lm}

\end{document}